\documentclass[11pt]{amsart}
\usepackage{amssymb}
\usepackage{amsfonts}
\usepackage{graphicx}
\usepackage{epsfig}
\usepackage{amssymb}
\usepackage{amsmath}
\usepackage{amsfonts}
\usepackage{amscd}

\newtheorem{lem}{Lemma}

\newtheorem{nott}[lem]{Notation}

\def\br{\mathbb{R}}
\def\bc{\mathbb{C}}
\def\bh{\mathbb{H}}

\def\Span{\mathrm{Span}}

\newcommand{\bF}{{\mathbb F}}

\newcommand{\bo}{{\mathbb O}}

\newcommand{\im}{\operatorname{Im}}

\renewcommand{\cos}{cos}

\newcommand{\p}{\partial}

\newcommand{\ra}{\rightarrow}
\newcommand{\lra}{\longrightarrow}

\newcommand{\bs}{\backslash}
\newcommand{\ov}{\overline}
\newcommand{\col}{\!:\!}

\theoremstyle{plain}
\newtheorem{theorem}{Theorem}[section]
\newtheorem{proposition}[theorem]{Proposition}
\newtheorem{lemma}[theorem]{Lemma}
\newtheorem{corollary}[theorem]{Corollary}
\theoremstyle{definition}

\newenvironment{thA}{\begin{trivlist}\item[]{\bf Theorem  A\ }\begin{em}}
{\end{em}\end{trivlist}}

\newcommand{\C}{\mathbb{C}}

\let\cal\mathcal
\newcommand{\bR}{{\mathbb R}}
\def\Span{\mathrm{Span}}

\newcommand{\ch}{H_{\bc}^}
\newcommand{\fh}{H_{\bF}^}
\newcommand{\hh}{H_{\bh}^}
\newcommand{\oh}{H_{\bo}^2}

\def\Z{\mathbb{Z}}
\def\Q{\mathbb{Q}}
\def\B{\mathbb{B}}
\def\H{\mathbb{H}}

\def\F{\mathbb{F}}
\def\P{\mathbb{P}}
\def\I{\text{Im}}
\def\R{\text{Re}}

\def\h{\mathcal{H}}

\begin{document}
\title[Equidistribution]
        {Counting, mixing and Equidistribution of horospheres in geometrically finite
        rank one locally symmetric manifolds}
        \author{Inkang Kim}
        \date{}
        \maketitle

\begin{abstract}
In this paper we study  the equidistribution of expanding
horospheres in infinite volume geometrically finite rank one locally
symmetric manifolds and apply it to the orbital counting problem in
Apollonian sphere packing.
\end{abstract}
 \footnotetext[3]{The author
gratefully acknowledges the partial support of NRF grant
(R01-2008-000-10052-0) and a warm support of IHES, IHP and HIM
during his stay.}
\section{Introduction}
In this paper, we study the equidistribution of expanding
horospheres in infinite volume geometrically finite locally
symmetric rank one manifolds with respect to Burger-Roblin measure.
As an application we apply it to the orbital counting of
geometrically finite groups.

A priori it is not clear how to count the growth of the number of
orbit points in $\F^{n+1}$ under the infinite co-volume group
$\Gamma\subset O_\F(n,1)$ where $O_\F(n,1)=SO(n,1),SU(n,1), Sp(n,1)$
depending on $\F=\br,\bc,\bh$. If $Q_\F(v_0)=0$ for a signature
$(n,1)$ quadratic form $Q_\F$, we give a quantitative estimate of
the asymptotic growth
$$\# \{v\in v_0 \Gamma|\ ||v||<K\}$$ for any norm
$||\cdot||$ on $\F^{n+1}$ with the control of the error term.
Controlling the error term is crucial to our application to the
counting of prime curvature spheres in Apollonian sphere packing.
This orbital counting theorem follows from the equidistribution of
expanding closed horospheres in geometrically finite groups. First
we show that for any $\psi\in C^\infty_c(\Gamma\backslash G)^K$, the
average of this function on a horosphere of height $y$ can be
explicitly estimated in terms of $L^2$-product of $\psi$ and
$\phi_0$, and $y^{D-\delta}$ where $\phi_0$ is a unique
eigenfunction of the Laplace operator with $L^2$-norm 1 of
eigenvalue $\delta(D-\delta)$, and $\delta$ is the critical exponent
of $\Gamma$, $D$ is the Hausdorff dimension of the ideal boundary of
the associated symmetric space. See Theorem \ref{equidistribution}.

The techniques involve the unitary representation theory, measure
theoretic approach in algebraic Lie groups, Patterson-Sullivan
measure on limit sets and some geometrical insights in rank one
space. We carry out the computation in  explicit coordinates,
so-called horospherical coordinates in rank one spaces. See sections
\ref{pre} and \ref{horo}. We outline necessary backgrounds as the
proof evolves in coming sections.

Let $X$ be a real, complex or quaternionic hyperbolic space with
curvature between $-4$ and $-1$, and $\Gamma \subset Iso(X)=G$ a
geometrically finite group whose critical exponent is $\delta> D/2$
 where $D$ is a Hausdorff dimension of $\partial
X$. Let $ G=KAN$ be
 a fixed Iwasawa decomposition introduced in section
\ref{pre}. In other words, $N$ is a (generalized) Heisenberg group,
$A$ is a one-parameter subgroup stabilizing a chosen geodesic
$(0,0,y)$ in horospherical coordinates (see section \ref{pre},
\ref{horo}), and $K$ is a maximal compact subgroup stabilizing
$(0,0,1)$.

The main theorem is:
\begin{theorem}  Let $\Gamma$ be a geometrically finite discrete subgroup in $SO(n,1)$, $SU(n,1),
Sp(n,1)$ with the critical exponent $\delta>\frac{D}{2}$ where $D$
is the Hausdorff dimension of the boundary  of the associated
symmetric spaces of the sectional curvature between $-4$ and $-1$.
Suppose $v_0$ is in the light cone such that $v_0\Gamma$ is
discrete, and the stabilizer of $v_0$ in $g_0^{-1}\Gamma g_0$ is in
$NM$. Then for any norm $||\cdot||$ on $\F^{n,1}$
$$\#\{v\in v_0\Gamma:||v||<T\}\sim c_{\phi_0}\delta^{-1}T^\delta\int_{ K} ||v_0(g_0^{-1}kg_0)||^{-\delta}dk.$$ If $||\cdot||$
is $g_0^{-1}Kg_0$-invariant, then
$$\#\{v\in
v_0\Gamma:||v||<T\}=c_{\phi_0}\delta^{-1}T^\delta||v_0||^{-\delta}(1+O(T^{-\delta'})).$$
Here $\delta'$ depends only on the spectral gap.
\end{theorem}
This theorem is a generalization of \cite{O} to a general rank one
case.\\
 {\bf Notation:}
$f(x)=O(x)$ implies that $\lim_{x\ra\infty}\frac{f(x)}{x}< \infty$.

{\bf Outline of a proof}: One introduces a continuous counting
function $F_T(g)=\sum_{\gamma\in (\Gamma\cap NM)\backslash \Gamma}
\chi_{B_T}(v_0 \gamma
 g),$ where $B_T=\{v\in v_0G: ||v||<T\}$. Specially $F_T(e)$ is the
 orbital counting function of $\Gamma$. Let
$\phi_\epsilon\in C_c^\infty(G)$ be a nonnegative function supported
on a small neighborhood $U_\epsilon$ of the identity element with
$\int_G \phi_\epsilon=1$. Define a function defined on
$\Gamma\backslash G$ (i.e., invariant under left $\Gamma$-action) by
$$\Phi_\epsilon(\Gamma g)=\sum_{\gamma\in
\Gamma}\phi_\epsilon(\gamma g).$$ One is interested to estimate
$$\langle F_T, \Phi_\epsilon\rangle=\int_G
\chi_{B_T}(v_0g)\Phi_\epsilon(g)dg$$$$=\int_{M\backslash K}
\int_{y>T^{-1}||v_0k||}\int_{n_x\in (\Gamma\cap NM)\backslash NM/M}
\psi_k(n_xa_y)dn y^{-D-1}dydk
$$ where $\psi_k$ is an average of $\Phi_\epsilon$ over a compact
group $M\subset K$ in Langland decomposition and $(\Gamma\cap
NM)\backslash NM/M$ denotes the projection into $\Gamma\backslash
G/M$. Geometrically $(\Gamma\cap NM)\backslash NM/M$ denotes the
quotient of a horosphere based at $\infty$ by the action of
$\Gamma\cap NM$. See Sections \ref{pre} and \ref{counting} for
details.

So it is important to estimate the average of a function over a
horosphere at height $0<y<1$, $\int_{n_x\in (\Gamma\cap
NM)\backslash NM/M} \psi(n_xa_y)dn$, for $\psi\in
L^2(\Gamma\backslash G/M)$. Here one uses the matrix coefficient
technique to draw Theorem \ref{equidistribution}
$$\int_{n_x\in (\Gamma\cap
NM)\backslash NM/M} \psi(n_xa_y)dn\sim y^{D-\delta}$$ for $\psi\in
L^2_c(\Gamma\backslash G)^K$, from which one deduces that
$F_T(e)\sim T^\delta$. The large part of the paper is devoted to
justify this.

First using a spectral gap, one can decompose $L^2(\Gamma\backslash
G)^K$ into $V_\delta$ where the bottom eigenfunction $\phi_0$ is a
unique $K$-invariant function, and the rest $V$ where $V$ does not
contain any complementary series $V_s,\ s\geq s_\Gamma$. So any
$\phi\in C^\infty(\Gamma\backslash G)^K\cap L^2(\Gamma\backslash
G/M)$ can be written as
 $$\phi=\langle \phi,\phi_0 \rangle \phi_0+\phi^\perp$$ and hence
$\langle a_y\phi_1, \phi_2\rangle= \langle \phi_1,
\phi_0\rangle\langle a_y\phi_0, \phi_2\rangle+ O(y^{D-s_\Gamma}
S_{2m}(\phi_1)S_{2m}(\phi_2))$. See Corollary  \ref{product}.

Secondly, we fix a positive $\eta\in C^\infty_c((\Gamma\cap
NM)\backslash NM/M)$ with $\eta=1$ on a bounded open set $B$ of
$(\Gamma\cap NM)\backslash NM/M$ and vanishes outside a small
neighborhood of $B$. Also for each $\epsilon<\epsilon_0$, let
$r_\epsilon$ be a nonnegative smooth function on $AN^-M$ whose
support is contained in
$$W_\epsilon=(U_\epsilon\cap A)(U_{\epsilon_0}\cap N^-)(U_{\epsilon_0}\cap
M)$$ and $\int_{W_\epsilon} r_\epsilon d\nu=1$. Finally define
$\rho_{\eta,\epsilon}$ on $\Gamma\backslash G$ which vanishes
outside $supp(\eta)U_{\epsilon_0}$ and for $g=n_xan^-m\in
supp(\eta)W_\epsilon$,
$$\rho_{\eta,\epsilon}(g)=\eta(n_x)r_\epsilon(an^-m).$$
This $\rho_{\eta,\epsilon}$ is introduced as a cut-off function to
estimate the average of any function $\phi$ in
$C^\infty_c(\Gamma\backslash G)^K$ over a horosphere at height $y$
in terms of $\epsilon$ and an $L^2$ inner product of $a_y\phi$ and
$\rho_{\eta,\epsilon}$. See Propositions \ref{basic} and
\ref{another}. Finally in Theorem \ref{equidistribution}, one
iterates the process to obtain the dominant term as the average of
$\phi_0$ over the horosphere at height $y$, which is $y^{D-\delta}$.

As an application to an Apollonian sphere packing in $\bR^3$, we
have
\begin{corollary}Given a bounded Apollonian sphere packing $\cal P$
in $\bR^3$, the number of spheres whose curvatures are less than $T$
grows asymptotically $\sim c T^\alpha$ for some positive $\alpha$
which is the Hausdorff dimension of the residual set.
\end{corollary}
Specially for asymptotic growth of the number of $k$-mutually
tangent spheres in $\br^3$ with prime curvatures, we obtain:
\begin{theorem}
Given a bounded primitive integral Apollonian sphere packing $\cal
P$ in $\bR^3$,  if $\pi^{\cal P}_k(T)$ denotes the number of
$k$-mutually tangent spheres whose curvatures are prime numbers less
than $T$, then
$$\pi^{\cal P}_k(T)\ll \frac{T^{\alpha}}{(\log T)^k}$$ for
$k\leq 5$.
\end{theorem}
Some of the lower bound seems to be known by experts like Sarnak
\cite{Sa, S1} for Apollonian circle packing. One can attempt a lower
bound for $r$-almost prime curvature spheres as follows. The
following corollary is a counting  of 5 spheres kissing each other
with certain properties, not the counting of individual sphere with
prime curvature. See the last section.
\begin{corollary}Given a bounded primitive integral Apollonian sphere packing $\cal
P$ in $\bR^3$, let $\pi_k^{\cal P}(T)^r$ denote the number of 5
spheres kissing each other (i.e. the number of orbits) among whose
at least $k$ curvatures are $r$-almost primes and all of whose
curvatures are less than $T$ where $r$ is a fixed positive integer
depending only on Apollonian packing. For any $k\leq 5$,
$$\frac{T^{\delta_\Gamma}}{(\log T)^k}\ll \pi_k^{\cal P}(T)^r.$$
\end{corollary}

For abundant literatures of this subject, see \cite{R1,R2} for
example.

{\bf Organization of the paper}: In sections 2 and 3, we give
preliminary backgrounds for rank one symmetric space and specially
introduce horospherical coordinates to calculate the Buseman
function explicitly. In section \ref{bottom}, we deal with the
bottom eigenfunction and its average over horospheres. This part is
essential for our orbital counting problem. In section
\ref{unitary}, we recall unitary representations of rank one
semisimple Lie groups and generalize Shalom's result on matrix
coefficents of spherical unitary representations of rank one group.
In section \ref{laplace} we give some consequences of unitary
representations using spectral gap theorem due to Hamenst{\"a}dt.
After section \ref{equid}, we closely follow the proofs of \cite{O}.
We record them for the reader's convenience. In the final section,
we are concerned with an integral Apollonian sphere packing, and
derive some asymptotic growth on the number of prime curvature
spheres using uniform spectral gap theorem due to
Bourgain-Gamburd-Sarnak, Varju-Salehi Gosefidy, Breuillard-Green-Tao
and Pyber-Szabo.
\section{Preliminaries}\label{pre}
The symmetric spaces of $\br$-rank one of non-compact type are the
hyperbolic spaces $\fh n$, where $\bF$ is either the real numbers
$\br$, or the complex numbers $\bc$, or the quaternions $\bh$, or
the Cayley numbers $\bo$; in the last case $n=2$. They are
respectively called as real, complex, quaternionic and octonionic
hyperbolic spaces (the latter one $\oh$ is also known as the Cayley
hyperbolic plane). Algebraically these spaces can be described as
the corresponding quotients: $SO(n,1)/SO(n)$, $SU(n,1)/SU(n)$,
$Sp(n,1)/Sp(n)$ and $F_4^{-20}/Spin (9)$ where the latter group
$F_4^{-20}$ of automorphisms of the Cayley plane $\oh$ is the real
form of $F_4$ of rank one.

Following Mostow \cite{Mo} and using the standard involution
(conjugation) in $\bF$, $z\to\bar z$, one can define projective
models of the hyperbolic spaces $\fh n$ as the set of negative lines
in the Hermitian vector space $\bF ^{n,1}$, with Hermitian structure
given by the indefinite $(n,1)$-form
$$
\langle z,w\rangle =z_1\overline w_1+\cdots+z_n\overline
w_n-z_{n+1}\overline w_{n+1}\,.
$$
 However, it does not work for the Cayley plane since $\bo$ is
non-associative, and one should use a Jordan algebra of $3\times 3$
Hermitian matrices with entries from $\bo$ whose group of
automorphisms is $F_4$, see \cite{Mo}.

Another models of $\fh n$ use the so called horospherical
coordinates  based on foliations of  $\fh n$ by horospheres centered
at a fixed point $\infty$ at infinity $\p\fh n$ which is
homeomorphic to $(n\dim_{\br}\bF-1)$-dimensional sphere. Such a
horosphere can be identified with the nilpotent group $N$ in the
Iwasawa decomposition $KAN$ of the automorphism group of $\fh n$.
The nilpotent group $N$ can be identified with the product
$\bF^{n-1}\times\im \bF$ (see \cite{Mo}) equipped with the
operations:
$$
(\xi,v)\cdot (\xi',v')=(\xi+\xi',v+v'-2\im \langle\langle \xi,\xi'
\rangle\rangle )\quad \text{and}\quad (\xi,v)^{-1}=(-\xi,-v)\,,
$$
where $\langle\langle ,\rangle\rangle$ is the standard Hermitian
product in $\bF^{n-1}$, $\langle\langle z,w\rangle\rangle =\sum
z_i\ov{w_i}$. The group $N$ is a 2-step nilpotent Carnot group with
center $\{0\}\times \im\bF \subset \bF ^{n-1}\times \im\bF$, and
acts on itself by the right translations $T_h(g)=g\cdot h\,, h,g\in
N$.

Now we may identify
$$
\fh n\cup \p \fh n\bs\{\infty\} \lra N \times [0,\infty)=
\bF^{n-1}\times \im\bF\times [0,\infty)\,,
$$
and call this identification the {\it ``upper half-space model"} for
$\fh n$ with the natural horospherical coordinates $(\xi,v,u)$. In
these coordinates, the above right action of $N$ on itself extends
to an isometric action (Carnot translations) on the $\bF$-hyperbolic
space in the following form:
$$T_{(\xi_0,v_0)} \col (\xi,v,u)\longmapsto (\xi_0+\xi\,,v_0+v-2\im
\langle\langle \xi,\xi_0\rangle\rangle \,,u)\,,
$$
where $(\xi,v,u)\in \bF^{n-1}\times \im\bF\times [0,\infty)$. Hence
rank one space $H^n_\F$ can be written  in horospherical coordinates
as
$$\h^n_\F=\{(z,t, y)|z\in\F^{n-1}, t\in \im\bF, 0<y\in \bR\}.$$ An hyperbolic isometry  fixing $0$ and
$\infty$ acts as
$$(z,t,y)\ra (r z O_\F, r^2 t, r^2 y)$$ where $O_\F$ is an element of
$M$. Here the group $MA$ fixing the origin of $N$ and $\infty$ is
$(U(n-1)\cdot U(1))\times\br$ in $\ch n$, $(Sp(n-1)\cdot
Sp(1))\times
{\mathbb R}$ in $\hh n$, and $Spin(7)\times\br
$ in $\oh$. Note that once the height $y$ is fixed, which is a
horosphere of $H^n_\bF$, $NM$ leaves invariant the horosphere, i.e.,
$NM$ acts on each orbit of $N$ by
$$ (O_\bF,(\chi,v,0)):(z,t,y)\ra (    (\chi+z)O_\bF, v+t-2\im \langle\langle  t,\chi\rangle\rangle, y)                      .$$
Later when we use the notation $(\Gamma \cap NM)\backslash NM/M$, it
denotes the projection into $\Gamma\backslash G/M$.  Since $G/M$ is
the unit tangent bundle $T^1(G/K)$, $(\Gamma \cap NM)\backslash
NM/M$ is the quotient of a horosphere under the action of $\Gamma
\cap NM$.

It is not difficult to show that $(dt-2\I \langle\langle z, dz
\rangle\rangle)^2+ \langle\langle dz, dz \rangle\rangle$ is
invariant under $NM$ where $\langle\langle\ , \rangle\rangle$ is the
standard positive definite Hermitian product and  $ANM$ is a Borel
subgroup fixing $\infty$. Since a metric on $H^n_\F$ can be written
as $g_y\oplus dy^2$ where $g_y$ is defined on a horosphere along a
geodesic ending $\infty$, one can give a metric \cite{KP}
$$ds^2=\frac {dy^2+  (dt-2\I \langle\langle z, dz
\rangle\rangle)^2+ 4y\langle\langle dz, dz \rangle\rangle}{y^2}.$$
This metric has the sectional curvature between $-1$ and
$-\frac{1}{4}$ so that the volume entropy for real hyperbolic
$n$-manifold is $(n-1)/2$, the volume entropy of complex hyperbolic
$n$-manifold is $n$ and the volume entropy of quaternionic
hyperbolic manifold is $2n+1$.

The volume form on $H^n_\F$ can be written as
\begin{eqnarray}
\frac{2^{n-1}}{y^{(n+1)/2}}dVol_z dy\ (\text{real\
hyperbolic})\end{eqnarray}
\begin{eqnarray}\label{volume}
\frac{4^{n-1}}{y^{n+1}} dVol_zdt dy\ (\text{complex\
hyperbolic}),\end{eqnarray}\begin{eqnarray}
\frac{16^{n-1}}{y^{2n+2}} dVol_zdt dy\ (\text{quaternionic \
hyperbolic})
\end{eqnarray}
 where $dVol_z$ is a
volume form on $\F^{n-1}$. The set $\{x=(z,t)\}$ is identified with
the Nilpotent group $N$ in Iwasawa decomposition $ANK$ of
$Iso(H^n_\F)$ and its right action is
$$(z,t,y)(w,s)=(z+w,t+s-2\text{Im}\langle\langle z,w\rangle\rangle,y).$$ The  ideal boundary at infinity of $H^n_\F$ is $N\cup \infty$. As
usual $A$ will be an one-dimensional group translating along
$\{(0,0,y)|y>0\}$ and its right action is given by
$(z,t,y)a_r=(rz,r^2t,r^2y)$ so that any point $(z,t,y)\in H^n_\F$ is
$$(0,0,1)(\frac{1}{\sqrt y}z,\frac{1}{ y}t)a_{\sqrt y}.$$ In this
way we will identify a point $(z,t,y)$ in $H^n_\F$ with
$[(\frac{1}{\sqrt y}z,\frac{1}{ y}t),a_{\sqrt y}]\in N \times A$ in
a fixed Iwasawa decomposition $KNA$.

Note that the metric $g_y\oplus dy^2$ is such that $
g_y=e^{2y}g_0\oplus e^{4y}g_0$ on $\F^{n-1}$ part and $\I \F$ part,
respectively,  $e^{-2y}g_y$ converges to nonriemannian metric,
Carnot-Caratheodory metric on $N$. It's distance is given by
$$d_N((x,t),(w,s))=|(x,t)(w,s)^{-1}|=(|x-w|^4+(t-s+2\I
\langle\langle x, w \rangle\rangle )^2)^{1/4}.$$ One can even define
$N$-invariant metric on horospherical model $\h^n$ by
$$d((x,t,y), (w,s,z))=(|x-w|^4+(t-s+2\I
\langle\langle x, w \rangle\rangle )^2+|y-z|^2)^{1/4}.$$

\section{Busemann function and Horosphere}\label{horo}
In this section we normalize the metric so that the sectional
curvature is between $-1$ and $-\frac{1}{4}$ and we fix a reference
point $o=(0,0,1)$. $H^n_\F$ can be realized a unit ball $\B^n$ in
$\F^n$. Two points $(0',-1)$ and $(0',1)$ will play a special role.
There is a natural map from $\B^n$ to $\P(\F^{n,1})$, where
$\F^{n,1}$ is equipped with
$$ \langle z,w\rangle =z_1\overline
w_1+\cdots+z_n\overline w_n-z_{n+1}\overline w_{n+1}\,.
$$ defined
as
$$(w',w_n)\ra (w',w_n,1).$$ From $\B^n$ to the horospherical model
$\h^n$, one define the coordinates change as
$$(z',z_n)\ra (\frac{z'}{1+z_n},\frac{2\I z_n}{|1+z_n|^2},\frac{1-|z_n|^2-|z'|^2}{|1+z_n|^2}).$$
It's  inverse from the horospherical model $\h^n$ to $\P\F^{n,1}$ is
given by
$$(\xi,v,u)=[(\xi,\frac{1-|\xi|^2-u+v}{2},\frac{1+|\xi|^2+u-v}{2})],$$
where $v$ is pure imaginary, i.e., $iv$ in complex case, and
$iv_i+jv_2+kv_3$ in quaternionic case. According to this coordinate
change, $(0',1)=[(0',1,1)]$ corresponds to the identity element
$(0',0)$ in Heisenberg group, $(0',-1)=[(0',-1,1)]$ to $\infty$, and
$(0',0)$ to $(0,0,1)$.

In the rest of the section, we carry out the calculations only in
complex hyperbolic space but it goes through the quaternionic case.
 If $x,y\in H^n_\C$ and $X,Y\in \C^{n,1}$ correspond to
$x,y$, the distance between them is
$$\cosh^2(\frac{d(x,y)}{2})=\frac{\langle X,Y\rangle \langle Y,X\rangle}{\langle X,X\rangle \langle Y,Y\rangle},$$ where
$\langle X,Y\rangle=\sum_{i=1}^n x_i\bar y_i-x_{n+1}\bar y_{n+1}$. A
Busemann function based at $\xi$ is defined as
$$B_\xi(z)=\lim_{t\ra \infty}(d(z,\gamma_t)-t)$$ where $\gamma_t$ is
a unit speed geodesic starting from $o$ and ending $\xi\in \partial
H^n_\C$. First we calculate the Busemann function based at
$\infty=(0',-1,1)$. $\gamma_t$ can be chosen as
$$\gamma(t)=(0',-\tanh (t/2), 1)\subset \C^{n,1}.$$ Then a straightforward calculation
shows that
$$d(z,\gamma_t)\ra t +\log \frac{|z_n+1|^2}{1-\langle\langle z, z\rangle\rangle},$$ where $z=(z',z_n)\in
\B^n$ and $\langle\langle z, w\rangle\rangle=\sum z_i\bar w_i$ is
the standard positive definite Hermitian product. We denote it as a
double bracket whereas $\langle\ , \ \rangle$ is $(n,1)$ Hermitian
product on $\C^{n,1}$. Then
$$B_\infty(z)=\log \frac{|z_n+1|^2}{1-\langle\langle z,
z\rangle\rangle}.$$ If we denote $Q=(0',-1,1)$ and $Z=(z,1)$ in
$\C^{n,1}$, it is easy to show that
\begin{eqnarray}\label{Busemann}
e^{-B_Q(z)}=\frac{-\langle Z,Z\rangle}{\langle Z, Q\rangle \langle
Q, Z \rangle}.
\end{eqnarray}

Using elliptic isometry fixing $o=(0,0,1)\in \h^n$, one can send $Q$
to any other point in the ideal boundary. Since it preserves
Busemann function and Hermitian inner product, above formula holds
for any $Z,Q$. Using this formula, it is easy to check that
$e^{-B_\infty(z)}=y$ for $z=(\xi,v,y)$ in horospherical coordinates.
This explains the last height coordinate in horospherical
coordinates.

We are interested in $e^{-B_Q(z)}$ for $Q\neq \infty$. Let
$Q=(\xi,v)$ and $z=(x,t,y)\in \h^n$. These will correspond in
$\C^{n,1}$ to $(x,\frac{1-|x|^2-y+it}{2},\frac{1+|x|^2+y-it}{2})$
and $(\xi,\frac{1-|\xi|^2+iv}{2},\frac{1+|\xi|^2-iv}{2})$
respectively. A direct calculation using equation (\ref{Busemann})
gives
$$e^{-B_Q(z)}=\frac{4y}{|2 \langle\langle x,
\xi\rangle\rangle-|\xi|^2-|x|^2-y+i(t-v)|^2}.$$ Using the relation
$\langle\langle x-\xi, x-\xi
\rangle\rangle=|x-\xi|^2=|x|^2+|\xi|^2-2\R \langle\langle x,
\xi\rangle\rangle$, above formula becomes
\begin{eqnarray}\label{distance}
e^{-B_Q(z)}=\frac{4y}{(|x-\xi|^2+y)^2+(t-v+2\I \langle\langle x,
\xi\rangle\rangle)^2}.
\end{eqnarray}
Note that the denominator is comparable with the distance defined in
section \ref{pre}.

\section{Opposite Nilpotent group}\label{opposite}
We fixed a Iwasawa decomposition $ANK$ so that $N$ is a 2-step
nilpotent group which is a Heisenberg group. To describe an opposite
Nilpotent group $N^-$ for later use, we introduce another
coordinates changes. To define a Hermitian form of signature $(n,1)$
one can equally use
$$\left[
\begin{matrix}
              0 & 0  & 1\\
              0 & I_n & 0 \\
              1 & 0 & 0 \end{matrix}\right]$$ for the product.
              Somehow this matrix simplifies the calculations the
              most. In this context, a coordinate change $\phi$ from the
              Horospherical model $\h^n$ to $\P(\F^{n,1})$ is given
              by
$$\phi(\zeta,v,u)=[(-|\zeta|^2-u+v)/2,\zeta,1]$$ see \cite{KP}.

 One can easily show that the right translation by $(\tau,t)$, i.e.,
 the map
 $(z,s)\ra (z,s)(\tau,t)$
 corresponds to matrix multiplication
$$\left[\begin{matrix} \frac{-|z|^2+s}{2}& z & 1 \end{matrix}\right]\left[\begin{matrix} 1& 0 & 0 \\
                      -\tau^* & I  & 0 \\
                     \frac{-|\tau|^2+t}{2} & \tau & 1
                     \end{matrix}\right].$$
   Since we consider right actions instead of left actions,
                       we can define the opposite
                       Nilpotent group $N^-$ as
$$N^-=\{\left[\begin{matrix} 1& 0 & 0 \\
                       -\tau^* & I  & 0\\
                       \frac{-|\tau|^2+t}{2} & \tau & 1
                       \end{matrix}\right]\}$$ whereas $N$ consists
                       of upper triangular matrices.
Then using the multiplication rules in section \ref{pre} one can
show that
$$n_x a_y=a_yn_{x/y}, n^-_x a_y=a_y n^-_{yx}.$$ This will be used in
section \ref{bottom}. In these coordinates, the origin of the
Heisenberg group corresponds to $(0,0,1)$ in $\P(\F^{n,1})$ and it
is the stabilized by the right action of $NM$. This fact will be
used in section \ref{counting}.

\section{Digression to a general Riemannian geometry}\label{digression} Let $M$ be a
Riemannian manifold, $u$ a function such that $|\nabla u|=1$. Then
\begin{proposition}
$$\triangle u(x)=\text{ mean\ curvature\ at\ x\ of}\
u^{-1}(u(x)).$$ If $u$ is a Busemann function on the rank one
symmetric space $X$, $u^{-1}(u(x))$ is a horosphere and $\triangle
u(x)=D$ where $D$ is the Hausdorff dimension of $\partial X$.
\end{proposition}\begin{proof}
$$\triangle u=div \nabla u=div \xi,\ \xi=\nabla u.$$
Since $|\xi|=1$, $\langle \nabla_v\xi,\xi\rangle=0$ for any $v\in
TM$. Also since $\nabla \xi\in End(TM)$ is symmetric (the
antisymmetric part of $\nabla\xi$ is $d(du)$),
$$\langle \nabla_v\xi , w\rangle=\langle \nabla_w\xi, v\rangle,$$
hence $\langle \nabla_\xi \xi, w\rangle=\langle
\nabla_w\xi,\xi\rangle=0$ for any $w\in TM$.  So $\nabla_\xi \xi=0$.

Let $e_i$ be orthonormal basis of $T_xM$ with $e_1=\xi$. Then
$$div\xi=-\sum_{i=1}^n \langle \nabla_{e_i} \xi,
e_i\rangle=-\sum_{i=2}^n\langle \nabla_{e_i} \xi,
e_i\rangle=-\sum_{i=2}^n II(e_i,e_i)$$ where $II$ is a second
fundamental form on $u^{-1}(u(x))$.  The second claim is well-known
\cite{BCG} (page 638-639).
\end{proof}
Note that for $f:\bR\ra \bR$ and $u:M\ra \bR$, using $div (g
V)=\langle \nabla g, V \rangle+ g div V$ for a smooth function $g$
and a vector field $V$ on $M$,
\begin{eqnarray}\label{st}
\triangle(f(u))=f''(u)|\nabla u|^2+f'(u)\triangle u,\end{eqnarray}
and if $|\nabla u|=1$, $\triangle(f(u))=f''(u)+f'(u)\triangle u$.

For example if $X$ is a symmetric space of rank one and $u$ is a
distance to a fixed hyperplane, $\triangle u=\phi(u)$ where $\phi$
is a solution to
$$\phi(u)=Tr(II(u))$$ where $II:\bR\ra (n-1)\times(n-1)\text{
symmetric\ matrices}$ satisfying $II'+II^2+ Riem=0$. Here
$TX=\F\nabla u\oplus(\F\nabla u)^\perp$ and $(\nabla
u)^\perp=(Im\F)\nabla u\oplus (\F\nabla u)^\perp$. $ Riem$ preserves
this decomposition and
$$Riem=-4 Id\ \text{on}\ (Im\F)\nabla u,$$
$$=-Id\ \text{on}\ (\F\nabla u)^\perp.$$

The solution of $II'+II^2+ Riem=0$ is
$$\phi_1(u)Id_{(Im\F)\nabla u}+\phi_2(u)Id_{(\F\nabla u)^\perp}$$
where $\phi_1'+\phi_1^2-4=0$ and $\phi_2'+\phi_2^2-1=0$. Then
$\phi_1(u)=\frac{2}{\tanh(2u)},\ \phi_2(u)=\tanh(u)$.

In conclusion
$$\triangle
u=\phi(u)=Tr(II(u))=(dim\F-1)\frac{2}{\tanh(2u)}+(dim\F)(n-1)\tanh(u).$$
For real hyperbolic 3-dimensional case, if $u$ is a distance to
$H^2_\bR$, $\triangle u=2\tanh(u)$. So if $\phi$ is an average of
the bottom eigenfunction over $H^2_\bR$ at the distance $u$,
$$\triangle \phi=\phi''(u)+\phi'(u) 2\tanh(u)$$
and since
$$\triangle \phi=-\delta(2-\delta)\phi$$
$\phi(u)$ grows asymptotically
$$ce^{-(2-\delta)u}+de^{-\delta u}.$$

If $u$ is a Busemann function on the symmetric space $X$ of rank
one, then by equation (\ref{st}) and $\triangle u=D$
$$\triangle(f(u))=f''(u)+Df'(u),$$ where $D$ is a Hausdorff
dimension of $\partial X$. Note that the solution space of the above
differential equation is generated by $e^{-c u}$ and $e^{-(D-c)u}$
for some $c$. Hence if $\phi_0^N$ is an average of the bottom
eigenfunction $\phi_0$ on a geometrically finite manifold
$\Gamma\backslash X$ over a horosphere at the distance $u$,
$$\phi_0^N(u)=c_{\phi_0}e^{-(D-\delta)u}+d_{\phi_0}e^{-\delta u}.$$
But in later sections, we will use $- \triangle$ as the Laplace
operator so that
$$\phi_0^N(u)=c_{\phi_0}e^{(D-\delta)u}+d_{\phi_0}e^{\delta u}.$$

In  section \ref{bottom}, we show that $c_{\phi_0}\neq 0$. If
$\Gamma_0\subset\Gamma$ is a subgroup of index $n$, then $1/\sqrt n
\tilde\phi_0$ is a unit $L^2$ norm eigenfunction of $-\triangle$
with eigenvalue $\delta(D-\delta)$ where $\tilde\phi_0$ denote the
lift to $\Gamma_0\backslash X$ of $\phi_0$ on $\Gamma\backslash X$.
Then the average of $1/\sqrt n\tilde\phi_0$ over the horosphere at
the distance $u$ is
$$\frac{1}{\sqrt n}\phi_0^N(u).$$  Specially the constant
$c_{\phi_0^{\Gamma_0}}$ is equal to
$\frac{1}{\sqrt{[\Gamma:\Gamma_0]}}c_{\phi_0}$. This will be used in
Section \ref{application}.

\section{Bottom eigenfunction $\phi_0$}\label{bottom}
Let $X$ be a rank one symmetric space and fix a origin $(0,0,1)=o\in
X$. Let $\Gamma\backslash X$ be  geometrically finite with a
critical exponent $\delta> D/2$ where $D$ is the Hausdorff dimension
of $\partial X$. Let $B_\zeta$ be a Busemann function based at
$\zeta$ normalized that $B_\zeta(o)=0$. The Patterson-Sullivan
measures $\theta_x,\ x\in X$ satisfies:\\
The measures $\theta_x$ and $\theta_y$ are mutually absolutely
continuous and
\begin{eqnarray}\label{Patterson1}
\frac{d\theta_x}{d\theta_y}(\zeta)=e^{-\delta(
B_\zeta(x)-B_\zeta(y))}
\end{eqnarray}
and for any $\gamma\in \Gamma$
\begin{eqnarray}\label{Patterson2}
\gamma_*\theta_x=\theta_{\gamma x}.
\end{eqnarray}

The function defined by
$$\phi_0(x,t,y)=\int_{\Lambda_\Gamma} e^{-\delta B_\theta(x,t,y)} d\theta$$ where $d\theta=d\theta_o$ is a fixed
Patterson-Sullivan measure associated to a fixed reference point
$o\in X$, descends to a positive $L^2$-function on $\Gamma\backslash
X$ whose eigenvalue with respect to the Laplace operator
$-\triangle$ is $\delta(D-\delta)$ where $D$ is the Hausdorff
dimension of $\partial X$. We always normalize it that its
$L^2$-norm is 1. The function $\phi_0$ is given by in horospherical
coordinates according to equation (\ref{distance})
$$\phi_0(x,t,y)=\int_{\Lambda_\Gamma} (\frac{4y}{(|x-\zeta|^2+y)^2+(t-v+2\I \langle\langle x,
\zeta\rangle\rangle)^2})^\delta d\theta(\zeta,v)$$
$$=4^\delta y^{-\delta}\int_{\Lambda_\Gamma} (\frac{1}{(\frac{|x-\zeta|^2+y}{y})^2+(\frac{t-v+2\I \langle\langle x,\zeta\rangle\rangle}{y})^2})^\delta
d\theta(\zeta,v).$$  Note that the formula for real hyperbolic case
seems a bit different from this one but it is due to the fact that
the $H^n_\bR\subset H^n_\F$ sits as a Klein model \cite{Gol}(not
Poincar\'e model) and that the curvature is $-1/4$ in this section.
After we normalize the metric back to between $-4$ and $-1$, all the
formulas will turn out right. See section \ref{unitary}.

Let
$$\phi_0^N(y)=\int_{(\Gamma\cap NM)\backslash NM/M}\phi_0(x,t,y) dn $$ be the average over a horophere where
$dn=dtdx$. This is independent of the choice of a fundamental domain
since $dn$ is $NM$-invariant and $\phi_0$ is $\Gamma$-invariant.

We will be interested in horospheres whose images are closed in
$\Gamma\backslash H^n_\F$. Specially we will consider horospheres
whose base points are either parabolic fixed points or points
outside the limit set of $\Gamma$, see Lemma \ref{closed}.
\begin{proposition}We identify $N$ with a horosphere based at $\infty$ (in fact, the orbit of $N$) and suppose $(\Gamma\cap NM)\backslash
NM/M$, which is a quotient of the horosphere by $\Gamma\cap NM$, has
a closed image in $\Gamma\backslash X$. Then $\phi_0^N(y)\gg
y^{D-\delta}$ for all $0<y\ll 1$ where $D$ is the Hausdorff
dimension of $\partial X$, i.e., $(n-1)/2, n, 2n+1$ for real,
complex and quaternionic hyperbolic space respectively.
\end{proposition}
\begin{proof} We carry out the calculation in complex hyperbolic
case and indicate the difference in the other two cases.\\
 CASE I)
If $\infty\notin \Lambda_\Gamma$, then $NM\cap \Gamma$ is trivial
and
$$\phi_0^N(y)=\int_N 4^\delta y^{-\delta}\int_{\Lambda_\Gamma} (\frac{1}{(\frac{|x-\zeta|^2+y}{y})^2+(\frac{t-v+2\I \langle\langle x,\zeta\rangle\rangle}{y})^2})^\delta
d\theta(\zeta,v) dn.$$ Change the variables to
$$z=\frac{x}{\sqrt y},s=\frac{t}{y},\ x\in \C^{n-1}, t\in \bR,$$
to get
$$dtdx=y^n ds dz$$ and
\begin{eqnarray}\label{estimate}
\phi_0^N(y)=4^\delta y^{n-\delta}\int_{\C^{n-1}\times \bR}\\
\int_{\Lambda_\Gamma} \frac{dsdz}{[(|z-\zeta'|^2+1)^2+|s-v'+2\I
\langle\langle z, \zeta' \rangle\rangle|^2]^\delta} d\theta.
\end{eqnarray}

Change $(w,t)=(z,s)(\zeta',v')^{-1}$. Then since $dn$ is
$N$-invariant
$$\phi_0^N(y)=4^\delta y^{n-\delta}\int_N\int_{\Lambda_\Gamma} \frac{dtdw}{[|w|^4+|t|^2+2|w|^2+1]^\delta} d\theta$$
$$=4^\delta y^{n-\delta}\int_N
\frac{dtdw}{[|w|^4+|t|^2+2|w|^2+1]^\delta}.$$ The equality holds
since $d\theta$ is a probability measure. One can show that the
integral converges for $\delta>\frac{n}{2}$ to a nonzero number. In
this calculation, note that for real hyperbolic case, there is no
$t$ factor, so from $z=x/\sqrt y$, $dx=(\sqrt y)^{n-1} dz$, this
gives $y^{\frac{n-1}{2}-\delta}$ in front.  In quaternionic case,
$dx=(\sqrt y)^{4n-4}, dt=y^3 ds$, to get $y^{2n+1-\delta}$ in front.
CASE II below is similar.


CASE II) $\infty\in \Lambda_\Gamma$\\
By the theorem of \cite{B}, $\infty$ is a bounded parabolic fixed
point, so $\Gamma\backslash (\Lambda_\Gamma- \infty)$ is a compact
set. Hence it is bounded. Let $F_\Lambda=(\Gamma\cap
NM)\backslash\Lambda_\Gamma\subset F_N=(\Gamma\cap NM)\backslash
NM/M$ be fundamental sets under the left action of $\Gamma\cap NM$.
Note that for a fixed $F_\Lambda\subset F_N$ and $X\in H^n_\F,
\zeta\in F_\Lambda$, by the property (\ref{Patterson2}) of
Patterson-Sullivan measure,
$$(n^{-1}_*\theta_X)(\zeta)=\theta_{n^{-1}X}(\zeta)=\theta_X(n\zeta)$$
for any $n\in NM\cap \Gamma$. Also by the property
(\ref{Patterson1}) of Patterson-Sullivan measure
$$\frac{d\theta_X}{d\theta_o}(n\zeta)=e^{-\delta B_{n\zeta}(X)},$$
$$\frac{d\theta_{n^{-1}X}}{d\theta_o}(\zeta)=e^{-\delta
B_\zeta(n^{-1}X)}.$$ This implies that for $\zeta\in F_\Lambda$
$$\int_{n
F_\Lambda} \int_{F_N} e^{-\delta B_{n\zeta}(x,t,y)} dn
d\theta(n\zeta)=\int_{ F_\Lambda} \int_{n^{-1}F_N} e^{-\delta
B_{\zeta}(n^{-1}(x,t,y))} dn d\theta(\zeta).$$

 Hence,
$$\phi_0^N(y)=\int_{F_N}\int_{\Lambda_\Gamma} e^{-\delta
B_{\theta}(x,t,y)} d\theta dn=\sum_{n\in \Gamma\cap NM}\int_{n
F_\Lambda} \int_{F_N} e^{-\delta B_{\theta}(x,t,y)} dn d\theta$$
$$=\sum_{n\in \Gamma\cap NM}\int_{
F_\Lambda} \int_{n^{-1}F_N} e^{-\delta B_{\theta}(x,t,y)} dn
d\theta=\int_{F_\Lambda}\int_N e^{-\delta B_{\theta}(x,t,y)} dn
d\theta.$$ In terms of the coordinates
\begin{eqnarray}\label{bounded}
\phi_0^N(y)= 4^\delta y^{n-\delta}\\
\int_{N}\int_{(\Gamma\cap NM)\backslash\Lambda_\Gamma}
\frac{dsdz}{[(|z-\zeta'|^2+1)^2+|s-v'+2\I \langle\langle z, \zeta'
\rangle\rangle|^2]^\delta} d\theta,
\end{eqnarray}
 where $(\zeta',v')\in
a_{\frac{1}{\sqrt y}}F_\Lambda$.
A similar estimates holds to conclude that
$$\phi_0^N(y)\gg y^{n-\delta}.$$
\end{proof}

Fix generators $v_1,\cdots,v_k$ in $NM/M$ corresponding to the axes
of screw motions of $NM\cap \Gamma$ so that $v_1,\cdots,v_k$
together with $v_{k+1},\cdots,v_{2n-1}$ are basis of $NM/M$. Denote
$N^\perp$ the subspace generated by $v_{k+1},\cdots,v_{2n-1}$. Let
$F_\Lambda\subset B\subset F_N$ be an open set such that
\begin{enumerate}
\item if $\infty\notin \Lambda_\Gamma$ then
$\epsilon_0(B)=\inf_{u\in F_\Lambda,x\in B^c}|x-u|_N>0$

\item if $\infty$ a bounded parabolic fixed point then
$\epsilon_0(B)=\inf_{u\in F_\Lambda,x\in B^c}|x-u|_{N^\perp}>0$
where $B^c=F_N\setminus B$.
\end{enumerate}
For such an open set $B$
\begin{proposition}\label{bound}If $\delta>\frac{D}{2}$, $\phi_0^N(y)=\int_B \phi_0(x,t,y)
dn+O_{\epsilon_0(B)}(y^\delta)$ and $\phi_0^N(y)=O(y^{D-\delta})$.
\end{proposition}
\begin{proof}We give  a proof in complex hyperbolic case but as in the previous Proposition, the other
cases are similar. When $\infty\notin \Lambda_\Gamma$, using
equation (\ref{estimate}), we want to estimate
$$4^\delta y^{n-\delta}\int_{B^c}\int_{\Lambda_\Gamma}
\frac{dsdz}{[(|z-\zeta'|^2+1)^2+|s-v'+2\I \langle\langle z, \zeta'
\rangle\rangle|^2]^\delta} d\theta.$$

Since
$$\frac{1}{[(|z-\zeta'|^2+1)^2+|s-v'+2\I \langle\langle z, \zeta'
\rangle\rangle|^2]^\delta}$$$$\leq \frac{1}{[|z-\zeta'|^4+|s-v'+2\I
\langle\langle z, \zeta' \rangle\rangle|^2]^\delta}
=\frac{1}{[d_N((z,s), (\zeta',v'))^4]^\delta}  ,$$
 Since $dn$ is invariant under $N$, the above
integral is
$$\leq 4^\delta y^{n-\delta} \int_{|w|_N\geq \epsilon_0/\sqrt y}
\frac{dn}{(|w|_N^4)^\delta} $$ where $w=(z,s)(\zeta',v')^{-1}$ in
Heisenberg group. Let $w=(x,t)$ and write $dtdx=r^{2n-3}drdt dS$
where $dS$ is a volume form on unit sphere in $\mathbb{R}^{2n-2}$.
Then
$$4^\delta y^{n-\delta} \int_{|w|_N\geq \epsilon_0/\sqrt y}
\frac{dn}{(|w|_N^4)^\delta} \leq 4^\delta y^{n-\delta}
\int_{(r^4+t^2)^{1/4}\geq \epsilon_0/\sqrt y}
\frac{r^{2n-3}drdtdS}{(r^4+ t^2)^\delta}$$
$$\leq C y^{n-\delta}\int_{r\geq \epsilon_0/\sqrt y}\int_{\sqrt t\geq  \epsilon_0/\sqrt y}\frac{r^{2n-3}drdt}{(r^4+t^2)^\delta}. $$
Letting $t=\tan\theta r^2$, and for $\delta>\frac{n}{2}$ it becomes
$$y^{n-\delta}\int_{r\geq \epsilon_0/\sqrt
y}\frac{r^{2n-1}dr}{(r^4)^\delta}\int_{\frac{y^{3/4}}{\epsilon_0^{2/3}}\geq
\tan \theta} \frac{ d\theta}{\cos^2\theta(1+\tan^2\theta)^\delta}\ll
y^{\delta}.$$
As before we have
$$\phi_0^N(y)=4^\delta y^{n-\delta}
\int_{r,t \geq 0} \frac{r^{2n-3}drdtdS}{(r^4+ t^2+1)^\delta}$$
$$= C y^{n-\delta}\int_{r\geq 0}
\int_{t\geq 0}\frac{r^{2n-3}drdt}{(r^4+2r^2+t^2+1)^\delta}. $$ Hence
if $4\delta-2-2n+3>1,$ (i.e. $\delta>\frac{n}{2}$) then by letting
$t=\tan\theta \sqrt{(r^4+2r^2+1)}$ the integrals converge to a
nonzero number to conclude that
$$\phi_0^N(y)=O(y^{n-\delta}).$$

When $\infty$ is a bounded parabolic fixed point, the similar
estimates holds. Using
$$\int_{B^c}\int_{\Lambda_\Gamma} e^{-\delta
B_{\theta}(x,t,y)} d\theta dn=\sum_{n\in \Gamma\cap NM}\int_{n
F_\Lambda} \int_{B^c} e^{-\delta B_{\theta}(x,t,y)} dn d\theta$$
$$=\sum_{n\in \Gamma\cap NM}\int_{
F_\Lambda} \int_{n^{-1}B^c} e^{-\delta B_{\theta}(x,t,y)} dn
d\theta=\int_{F_\Lambda}\int_{\cup_{n\in \Gamma\cap NM}nB^c}
e^{-\delta B_{\theta}(x,t,y)} dn d\theta,$$  and following the
equation (\ref{bounded}), we want to estimate
$$
4^\delta y^{n-\delta}\int_{\cup_{n\in \Gamma\cap
NM}nB^c}\int_{(\Gamma\cap NM)\backslash\Lambda_\Gamma}
\frac{dsdz}{[(|z-\zeta'|^2+1)^2+|s-v'+2\I \langle\langle z, \zeta'
\rangle\rangle|^2]^\delta} d\theta,$$ where $(\zeta',v')\in
a_{\frac{1}{\sqrt y}}F_\Lambda$. The same estimation as in
$\infty\notin \Lambda_\Gamma$ gives
$$\int_{B^c}\int_{\Lambda_\Gamma} e^{-\delta
B_{\theta}(x,t,y)} d\theta dn\ll y^\delta,$$ to get
$$\phi_0^N(y)=\int_{B}\int_{\Lambda_\Gamma} e^{-\delta
B_{\theta}(x,t,y)} d\theta dn + \int_{B^c}\int_{\Lambda_\Gamma}
e^{-\delta B_{\theta}(x,t,y)} d\theta dn$$
$$\ll\int_{B}\int_{\Lambda_\Gamma} e^{-\delta
B_{\theta}(x,t,y)} d\theta dn + y^\delta.$$ Also in this case the
similar estimates give that for $\delta>\frac{n}{2}$
$$\phi_0^N(y)=O(y^{n-\delta}).$$
\end{proof}

We fixed Iwasawa decomposition $ANK$ so that $N$ is a Heisenberg
group.

 Let $N^-$ be the opposite Nilpotent group to $N$ so that
$$N\times A\times N^- \times M\ra G$$ is a diffeomorphism around a
neighborhood of $e$ and $d\nu$ is a smooth measure on $AN^-M$ so
that $dn\otimes d\nu$ is a Haar measure $d\mu$ on $G$. Fix a left
invariant metric $d_G$ on $G$ and $U_\epsilon$ is an
$\epsilon$-neighborhood of $e$ in $G$. Since $A\times N\times K\ra
G$ is a diffeomorphism and hence a bi-Lipschitz map around the
neighborhood of $e$,  there exists $l>0$ such that $U_\epsilon$ is
contained in $A_{l\epsilon}N_{l\epsilon}K_{l\epsilon}$ once we fix
some $\epsilon_0$ and take $\epsilon \leq \epsilon_0$. We fix a
positive $\eta\in C^\infty_c((NM\cap\Gamma)\backslash NM/M)$ with
$\eta=1$ on a bounded open set $B$ of $F_N$ so that
$\epsilon_0(B)>0$ as in the previous Proposition \ref{bound} and
vanishes outside a small neighborhood of $B$ so that
$$\phi_0^N(y)=\int_{B} \phi_0(n,y)\eta(n) dn+O(y^\delta).$$

Shirinking $\epsilon_0$ if necessary, we further assume that
$$supp(\eta)\times (U_{\epsilon_0}\cap AN^-M)\ra \Gamma\backslash G$$ is a
bijection to its image.

For each $\epsilon<\epsilon_0$, let $r_\epsilon$ be a nonnegative
smooth function on $AN^-M$ whose support is contained in
$$W_\epsilon=(U_\epsilon\cap A)(U_{\epsilon_0}\cap N^-)(U_{\epsilon_0}\cap M)$$ and $\int_{W_\epsilon} r_\epsilon d\nu=1$. Finally define
$\rho_{\eta,\epsilon}$ on $\Gamma\backslash G$ which vanishes
outside $supp(\eta)U_{\epsilon_0}$ and for $g=n_xan^-m\in
supp(\eta)W_\epsilon$,
$$\rho_{\eta,\epsilon}(g)=\eta(n_x)r_\epsilon(an^-m).$$ Then
\begin{proposition}\label{basic}For small $\epsilon\ll \epsilon_0$ and for
$y<1$, regarding $\phi_0\in C(\Gamma\backslash G)^K$
$$\phi_0^N(y)=\langle a_y\phi_0,\rho_{\eta,\epsilon}
\rangle_{L^2}+O(\epsilon y^{D-\delta})+O(y^\delta).$$
\end{proposition}
\begin{proof}As usual we give a proof only in complex hyperbolic case.
Since $$\langle a_y\phi_0,\rho_{\eta,\epsilon}
\rangle_{L^2}=\int_{W_\epsilon}
r_\epsilon(h)\int_{(NM\cap\Gamma)\backslash NM/M}
\phi_0(nha_y)\eta(n) dn d\nu(h)$$$$=\int_{(NM\cap\Gamma)\backslash
NM/M} \phi_0(nha_y)\eta(n) dn $$ we need to estimate $\phi_0(nha_y)$
for $n\in F_N$ and $h\in W_\epsilon$.

For $h=a_{y_0}n_x^-m\in W_\epsilon$ so that $|y_0-1|=O(\epsilon)$,
$$nha_y=na_{yy_0}n^-_{yx}m.$$
Since $n^-_{yx}=a_{y_1}n_{x_1}k_1\in A_{ly\epsilon}N_{ly\epsilon}K$
so that $|y_1-1|=O(y\epsilon)$,
$$nha_y=nn_{x_1yy_0y_1}a_{yy_0y_1}k_1m.$$
Since $\phi_0$ is $K$-invariant and $dn$ is $N$-invariant,
$$\int_{F_N} \phi_0(nha_y) \eta(n)dn=\int_{F_N}
\phi_0(nn_{x_1yy_0y_1}a_{yy_0y_1})\eta(n)dn$$
$$=\int_{F_N} \phi_0(na_{yy_0y_1})(\eta(n)+O_\eta(\epsilon))dn$$ by letting
$n'=nn_{x_1yy_0y_1}$ and so $\eta(n)=\eta(n'(n_{x_1yy_0y_1})^{-1})$
and as $|\eta(n')-\eta(n'(n_{x_1yy_0y_1})^{-1})|=O_\eta(\epsilon)$.
By Proposition \ref{bound},
$$\int_{F_N} \phi_0(nha_y) \eta(n)dn=\int_{F_N}
\phi_0(na_{yy_0y_1})\eta(n)dn+O_\eta(\epsilon
\phi_0^N(a_{yy_0y_1}))$$
$$=\phi_0^N(yy_0y_1)+O((yy_0y_1)^\delta)+O_\eta(\epsilon \phi_0^N(a_{yy_0y_1})).$$

Since by Proposition \ref{bound},
$$\lim_{y\ra 0}\frac{\phi_0^N(yy_0y_1)}{\phi_0^N(y)}=\lim_{y\ra
0}\frac{\phi_0^N(yy_0y_1)}{(yy_0y_1)^{D-\delta}}\frac{(yy_0y_1)^{D-\delta}}{y^{D-\delta}}
\frac{y^{D-\delta}}{\phi_0^N(y)}$$$$=\lim_{y\ra
0}(y_0y_1)^{D-\delta}=(1+O(\epsilon)),$$ and since
$\phi_0^N(a_{yy_0y_1})=O((yy_0y_1)^{D-\delta})$ we get
$$\int_{F_N} \phi_0(nha_y) \eta(n)dn=\phi_0^N(y)+O(y^\delta)+O_\eta(\epsilon y^{D-\delta}).$$
\end{proof}

Let $\{Z_1,\cdots,Z_k\}$ denote an orthonormal basis of the Lie
algebra of $G$ and $\Gamma\subset G$ a discrete subgroup. For $f\in
C^\infty(\Gamma\backslash G)^K\cap L^2(\Gamma\backslash G)$, one
considers the Sobolev norm $S_m(f)$:
$$S_m(f)=\max\{||Z_{i_1}\cdots Z_{i_m}(f)||: 1\leq i_j \leq k\}.$$
The following is standard. For $\phi\in C_c^\infty(\Gamma\backslash
G)^K$, there exists $\phi'\in C_c^\infty(\Gamma\backslash G)^K$ so
that
\begin{enumerate}
\item for small $\epsilon>0$ and $h\in U_\epsilon$,
$$|\phi(g)-\phi(gh)|\leq \epsilon \phi'(g)$$ for any $g\in \Gamma\backslash G$.
\item by Sobolev embedding theorem, there exists $q$ so that $S_m(\phi')\ll S_q(\phi)$ for each $m$, where the implied
constant depends only on ${\text{supp}(\phi)}$.
\end{enumerate}
\begin{proposition}\label{another}
Let $\phi\in C^\infty(\Gamma\backslash G)^K$. Then for any $0<y<1$
and any small $\epsilon>0$,
$$|I_\eta(\phi)(a_y)-\langle a_y\phi,
\rho_{\eta,\epsilon}\rangle|\ll(\epsilon+y) I_\eta(\phi')(a_y)$$
where $I_\eta(\phi)(a_y)=\int \phi(n a_y)\eta(n) dn$ and $\eta\in
C_c((NM\cap\Gamma)\backslash NM/M)$.
\end{proposition}
\begin{proof}For $h=an^-m\in W_\epsilon$,
$an^-ma_y=a_ya(a_{y^{-1}}n^-a_y)m$. Hence
$nan^-ma_y=na_y(aa_{y^{-1}}n^-a_ym)$ where $a_{y^{-1}}n^-a_y\in
U_{y\epsilon_0}\cap N^-$. Then as $\phi$ is $K$-invariant,
$$|\phi(na_y)-\phi(nha_y)|=|\phi(na_y)-\phi(na_yh')|\ll \phi'(na_y)(\epsilon+y\epsilon_0)$$
where $h'=aa_{y^{-1}}n^-a_y\in (U_\epsilon\cap
A)(U_{y\epsilon_0}\cap N^-)$. Hence
$$|\phi(na_y)-\int_{h\in W_\epsilon}\phi(nha_y)r_\epsilon(h)d\nu(h)|\ll \phi'(na_y)(\epsilon+y\epsilon_0).$$
By integrating on $(NM\cap\Gamma)\backslash NM/M$ we obtain
$$|I_\eta(\phi)(a_y)-\langle a_y\phi,
\rho_{\eta,\epsilon}\rangle_{L^2(\Gamma\backslash G)}|\ll
(\epsilon+y\epsilon_0) I_\eta(\phi')(a_y).$$ Since $\epsilon_0$ is
fixed, we get the desired result.
\end{proof}

\section{Unitary representations of rank one semisimple group}\label{unitary}
From this section, we normalize the metric so that its sectional
curvature is between $-4$ and $-1$. Equivalently we have to multiply
the metric tensors $ds^2$ by $1/4$. Then the volume form for complex
hyperbolic $n$-space is multiplied by $2^{-2n}$. Then the distance
from $(0,0,1)$ to $(0,0,y)$ will become $\log \sqrt y$. So by
changing the variable $\sqrt y=w$ and abusing the notation by
putting $w$ back to $y$, the volume form for complex hyperbolic
$n$-space is
\begin{eqnarray}\label{volume2}
\frac{1}{2y^{2n+1}} dVol_zdt dy \end{eqnarray} For real hyperbolic
$n$-space,
\begin{eqnarray}
\frac{1}{y^{n}} dVol_z dy \end{eqnarray} For quaternionic hyperbolic
$n$-space
\begin{eqnarray}
\frac{1}{8y^{4n+3}} dVol_zdt dy \end{eqnarray}
 where $dVol_z$ is a
volume form on $\F^{n-1}$. Also in all the formulas in section
\ref{bottom}, $n$ should be read as $2n$ under this normalization.

  Let
$\eta=\left[\begin{matrix}
   0 & 0 & 0 \\
   0 & 0 & 1 \\
   0 & 1 & 0\end{matrix}\right]$ be in $\mathfrak{so}(n,1),\mathfrak{su}(n,1),\mathfrak{sp}(n,1)$ so that
$e^{t\eta}$ is the 1-parameter subgroup constituting $A$ in $KAN$. A
direct calculation shows that
$$\mathfrak{g}_k=Ker(ad\eta-kId),\ k=0,\pm 1,\pm 2$$ are only root
spaces and there are only two positive roots $\beta,2\beta$ (in real
hyperbolic case, $2\beta$ is not a root) so that $\beta(\eta)=1$.
Note that real dimension of $\mathfrak{g}_1$ is $n-1,2n-2, 4n-4 $
resp. and that of $\mathfrak{g}_2=0,1, 3$ resp. Then the half sum
$\rho$ of the positive roots is
$$\rho=\frac{D}{2}\beta.$$
Since we will work with  horospheres based at $\infty$ and expanding
ones as $y\ra 0$, we will define a positive Weyl chamber by
$$A^+=\{a_y|0<y<1\}$$ as a multiplicative group. So $-\beta,-2\beta$ will be positive roots in this
paper, which will change the plus sign to the minus in all formulas
in the literature.

 In this section, we
prove the following. The kind of estimate we look for was first
established by Cowling, Haagerup and Howe for tempered
representations (i.e. unitary representations weakly contained in
the regular representation). A unitary representation
$(V,\pi=\int_{x\in\hat G}\pi_x d\mu(x))$ of $G$ is called tempered
if one of the following is true, see for example \cite{c,oh}.
\begin{enumerate}
\item For any $K$-finite unit vectors $v$ and $w$(i.e., the dimension of the subspaces spanned by $Kv$ and $Kw$ is finite), $$|\langle
\pi(g)v,w\rangle|\leq (\text{dim}\langle Kv\rangle \text{dim}\langle
Kw\rangle )^{1/2}\Xi_G(g)$$ for any $g\in G$ where $\Xi_G$ is the
Harish-Chandra function on $G$.
\item For almost all $x\in\hat G$, the irreducible representation
$\pi_x$ is strongly $L^{2+\epsilon}(G/Z(G))$, i.e., for any $p>2$,
there exists a dense subset $W\subset V$ such that for any $v,w\in
W$, the matrix coefficient $g\ra \langle \pi_x(g)v, w\rangle$ lies
in $L^p(G/Z(G))$.
\item For almost all $x\in \hat G$, $\pi_x$ is tempered in the sense
of (1).
\item $\pi$ is weakly contained in the regular representation
$L^2(G)$, i.e., any diagonal matrix coefficients of $\pi$ can be
approximated, uniformly on compact sets, by convex combinations of
diagonal matrix coefficients of the regular representation $L^2(G)$.
\end{enumerate}

 The following theorem is due to Y.
Shalom, \cite{Sha}, Theorem 2.1 p. 125 in case $G=SO(n,1)$ or
$SU(n,1)$. We generalize it to other rank one groups. For the
notational simplicity, we fix the notations first.

Let $G$ be a $\bR$-rank one simple Lie group. Pick an Iwasawa
decomposition $G=KAN$. Let $\lambda\in\mathfrak{a}'_{\bc}$ be a
complex linear form on the Lie algebra of $A$. This gives rise to
the character
$$
a\mapsto a^{\lambda}=e^{\lambda(\log(a))}
$$
on $A$. Let $M$ be the centralizer of $A$ in $K$. Let $Z^{\lambda}$
denote the space of $K$-finite complex valued functions on $G$ such
that $f(gman)=a^{-\lambda}f(g)$ for all $a\in A$, $m\in M$ and $n\in
N$ (notation taken from \cite{Ko}). $G$ acts on $Z^{\lambda}$ by
$(\pi_{\lambda}(g)u)(h)=u(g^{-1}h)$, for $g$, $h\in G$.

\begin{nott}
\label{s} Let $\rho$ denote the half-sum of positive roots. We write
$\lambda=(1+s)\rho$ with $s\in\bc$, $Z_{s}=Z^{\lambda}$, $\pi_s
=\pi_{\lambda}$.
\end{nott}

A description of the set of $s\in \C$ such that $Z_s$ is irreducible
and admits a $G$-invariant inner product can be found in \cite{Ko},
Theorem 10 page 641. It is the union of the line $\{\R(s)=0\}$, and
of the closed {\em critical interval} $\overline{CI}=[-s_1 (G),s_1
(G)]$, a symmetric interval on the real line. It is shown in
\cite{Ko}, that for $0\leq \lambda \leq \rho$, $\pi_\lambda$ and
$\pi_{2\rho-\lambda}$ are equivalent.
   The representations obtained for $\{\R(s)=0\}$ are called {\em
spherical principal series} representations of $G$. The
representations obtained when $s\in \overline{CI}$ are called {\em
spherical complementary series} representations of $G$.

The word {\em spherical} refers to the fact that these
representations contain a unique $K$-invariant line, generated by
the function $v_{\lambda}$ which equals $1$ on $K$.

According to Harish-Chandra, \cite{H}, every irreducible unitary
representation of $G$ which contains a nonzero $K$-invariant vector
is isomorphic to a representation from the spherical principal or
complementary series (except possibly for the trivial
representation).

\label{decay} Let $G$ be a $\bR$-rank one simple Lie group. Let
$(\pi_{s},Z_{s})$, $s\in[0,s_1 (G)]$, be a representation belonging
to the spherical complementary series of $G$. Pick a unit
$K$-invariant vector $v_{s}$ in $Z_{s}$. Define the function
\begin{eqnarray*}
 \Xi_{s} (g)=\langle\pi_{s} (g)v_{s},v_{s}\rangle_{s} .
\end{eqnarray*}
When $s=1$, $\lambda=2\rho$, hence $\pi_{2\rho}$ is equivalent to
$\pi_{0}$. So $\Xi_1=\Xi_G$ Harish-Chandra function on $G$.

 If $KA^+K$ is a polar decomposition of
$SU(n,1)$, then the Haar measure on it is (\cite{Sha})
   $$(e^{-2\beta(\log a)}-e^{2\beta(\log a)})(e^{-\beta(\log a)}-e^{\beta(\log a)})^{2n-2} dk da dk$$
where $\log:G\ra \mathfrak{g}$ is the inverse map of the exponential
map and $a(g)$ is the component of $A^+$ in the polar decomposition
of $g$. Note that in this formula, $A^+$ is regarded as an additive
group.
   Then for $y<1$, the measure is
comparable to $e^{-2n\beta(\log a)}dkdadk$,   that if we use $\log
y=a\in A^+$ to make $A^+$ a multiplicative group so that
$da=\frac{dy}{y}$ it is
\begin{eqnarray}\label{Haar}
   \frac{1}{y^{2n+1}}dkdydk.
   \end{eqnarray}  We hope that this switch between multiplicative
     and additive group does not cause any confusion to the reader.
This is consistent with the volume form (\ref{volume2}) on $H^n_\C$.
For real hyperbolic space, Haar measure is
$$(e^{-\beta(\log a)}-e^{\beta(\log a)})^{n-1}dkdadk$$ and comparable to
$$ e^{-(n-1)\beta(\log a)}dkdada  $$ for $y<1$, hence after put $\log
y=a$, the Haar measure on $KA^+K$ is
$$\frac{1}{y^n}dkdydk,   $$ which is comparable. In any case one can
write the Haar measure on $KA^+K$ as
\begin{eqnarray}\label{Haarmeasure}
\frac{1}{y^{D+1}}dkdydk
\end{eqnarray} for $y<1$.
For quaternionic hyperbolic case is the same.

 For $\lambda\in
i\mathfrak{a}^*$, and for $\log a(g)\leq 0$, the principal series
$\pi_\lambda$ has matrix coefficient decaying rates for
$K$-invariant unit vector $v_\lambda$,
$$|\langle \pi_\lambda(g) v_\lambda,v_\lambda\rangle|\leq\Xi
(g)\leq C|1-\beta(\log a(g)|e^{\frac{D}{2}\beta(\log a(g))}.$$ See
\cite{GV} section 4.6.4. For $0\leq \lambda\leq \frac{D}{2}\beta\in
\mathfrak{a}^*$, and for $\log a(g)\leq 0$,
 the complementary series representation
has matrix coefficient decaying rates
\begin{eqnarray}\label{comp}
\leq C|1-\beta(\log a(g)|e^{(-\lambda+\frac{D}{2}\beta)(\log a(g))}.
\end{eqnarray}
See \cite{Sha} (equation (10) page 132) also.

\begin{proposition}\label{est}For
all $g\in G$, there exist $c(G),C(G)$ such that
\begin{eqnarray}
c(G)\Xi_{1}(g)^{s}\leq \Xi_{s} (g)\leq C(G)(1-\log \Xi_{1}(g))
\Xi_{1}(g)^{s}. \label{gv}
\end{eqnarray}
\end{proposition}
\begin{proof}
Let $\lambda=(1+s)\rho,\ 0\leq s\leq s_1(G)$. Then $\pi_\lambda$ is
equivalent to $\pi_{2\rho-\lambda}$ where
$2\rho-\lambda=\rho-s\rho$. In view of equation (\ref{comp}),
$$\Xi_s(g) \leq C|1-\beta(\log a(g)|e^{(-2\rho+\lambda+\rho)(\log
a(g))}=C|1-\beta(\log a(g)|e^{(s\rho)(\log a(g))}.$$ But it is known
\cite{GV} (Theorem 4.6.5) that
$$\Xi_1(a)=\Xi(a)\geq e^{\rho(\log a)},\ a\in A^+.$$
Hence we get
$$\Xi_s(g) \leq C(G)(1-\log \Xi_{1}(g))
\Xi_{1}(g)^{s}.$$

For the lower bound, in \cite{GV} (4.7 (4.7.13)), \cite{Sha}
(equation (11) page 132), it is shown that
$$C'e^{(\rho-\lambda)\log a(g)}\leq \Xi_\lambda(g),\ 0\leq \lambda\leq
\rho.$$ Hence again using the fact $\pi_\lambda$ is equivalent to
$\pi_{2\rho-\lambda}$, we get the desired lower bound.
\end{proof}

\begin{theorem}\label{tem}
There exists a constant $C(G)$ such that for all $K$-finite vectors
$u$, $v\in Z_s=Z^{\lambda}$,
\begin{eqnarray*}
\langle\pi_{s} (g)u,v\rangle_{s}  \leq C(G)\,
(\dim\Span(Ku)\dim\Span(Kv))^{1/2} |u|_{s}\,|v|_{s} \,\Xi_{s} (g).
\end{eqnarray*}
\end{theorem}
\begin{proof}
The main ingredient in the proof is Cowling, Haagerup and Howe's
temperedness criterion \cite{c}. The following statement is a
combination of their Theorems 1 and 2.

\begin{proposition}
\label{chh} Let $\pi$ be a unitary representation of a semi-simple
algebraic group $G$. Then the following are equivalent.
\begin{enumerate}
  \item $\pi$ has a dense set of vectors whose coefficients  belong to $L^{2+\epsilon}(G)$ for all $\epsilon>0$.
  \item for all $K$-finite vectors $u$ and $v$ in $\pi$,
\begin{eqnarray}\hskip 1 cm \langle\pi(g)u,v\rangle  \leq C\,
(\dim\Span(Ku)\dim\Span(Kv))^{1/2} |u|\,|v| \,\Xi_{1}(g)
\label{inchh}
\end{eqnarray}where
$\Xi_1=\Xi$ is the Harish-Chandra function on $G$.
\end{enumerate}
\end{proposition}

Also, by Proposition \ref{est}, for all $g\in G$,
\begin{eqnarray}
c(G)\Xi_{1}(g)^{s}\leq \Xi_{s} (g)\leq C(G)(1-\log \Xi_{1}(g))
\Xi_{1}(g)^{s}. \label{gv}
\end{eqnarray}

Shalom's trick consists in tensoring representations until they
become almost square integrable. Let $s\in[0,s_1 (G)]$. For
$G=Sp(n,1)$, $n\geq 2$, $s_1 (G)=\frac{4n-2}{4n+2}>\frac{1}{2}$.
Thus there exists $t \in[0,s_1 (G)]$ such that $1-s-t \in[0,s_1
(G)]$. For $G=F_{4}^{-20}$, $s_1 (G)=\frac{10}{22}>\frac{1}{3}$.
Thus there exists $t \in[0,s_1 (G)]$ such that $1-s-2t \in[0,s_1
(G)]$. Consider the unitary representation
\begin{eqnarray*}
\pi=\pi_{s}\otimes\pi_{t}\otimes\pi_{1-s-t} \quad \textrm{(resp.
}\pi=\pi_{s}\otimes\pi_{t}\otimes\pi_{t}\otimes\pi_{1-s-2t}\textrm{).}
\end{eqnarray*}
Since all $\pi_{s}$ are irreducible, the translates of $v_{s}$
generate a dense subspace of $Z_{s}$. Let $u_1$, $v_1$ belong to
this subspace. Then $\langle\pi_{s} (g)u_1 ,v_1
\rangle_{s}=O(\Xi_{s}(g))$. Similarly, pick $u_2$, $v_2$ (resp.
$u_3$, $v_3$, $u_4$, $v_4$) in the vector space generated by
translates of $v_t$ (resp. of $v_{1-s-t}$, $v_{1-s-2t}$). Let
$U=u_{1}\otimes u_{2}\otimes u_{3}$, $V=v_{1}\otimes v_{2}\otimes
v_{3}$ (resp. $u_{1}\otimes u_{2}\otimes u_{3}\otimes u_{4}$
etc...). Then
\begin{eqnarray*}
\langle\pi(g)U,V\rangle &=&O(\Xi_{s}(g)\Xi_{t}(g)\Xi_{1-s-t}(g)) \quad \textrm{(resp. } O(\Xi_{s}(g)\Xi_{t}(g)\Xi_{t}(g)\Xi_{1-s-2t}(g))\textrm{)}\\
&=&O((1-\log\Xi_{1}(g))^{4}\Xi_{1}(g)).
\end{eqnarray*}
Since $y^{\delta}\log y\ra 0$ as $y\ra 0$ for any $\delta>0,$ for
$a_y,y<1$ we have $(1-\log\Xi_{1}(g))^{4}\Xi_{1}(g) \leq C
y^{\rho-\delta}$ for any $\delta>0$. Since the Haar measure is given
by
$$\frac{1}{y^{D+1}}dkdydk,$$ for small $\delta$
$$\int_K \int_0^1 (y^{\rho-\delta})^{2+\epsilon}
\frac{1}{y^{D+1}}dkdydk <\infty,$$
 thus belongs to $L^{2+\epsilon}(G)$
for all $\epsilon>0$.

Since these products generate a dense subspace of the tensor
product, Proposition \ref{chh} applies, $\pi$ is tempered, and
inequality (\ref{inchh}) holds for all $K$-finite vectors in the
tensor product.  Let $u$, $v\in Z_s$ be $K$-finite vectors. Since
\begin{eqnarray*}
\langle\pi(g)u\otimes v_t \otimes v_{1-s-t},v\otimes v_t \otimes
v_{1-s-t}\rangle=\langle\pi_s(g)u,v\rangle_s
\Xi_{t}(g)\Xi_{1-s-t}(g)
\end{eqnarray*}
and
\begin{eqnarray*}
|u\otimes v_t \otimes v_{1-s-t}|=|u|_s , \quad |v\otimes v_t \otimes
v_{1-s-t}|=|v|_s ,
\end{eqnarray*}
\begin{eqnarray*}
\langle \pi_s(g)u,v\rangle_s  \leq C\,
(\dim\Span(Ku)\dim\Span(Kv))^{1/2} |u|_s \,|v|_s
\,\frac{\Xi_{1}(g)}{\Xi_{t}(g)\Xi_{1-s-t}(g)}.
\end{eqnarray*}
Thanks to the lower bound (\ref{gv}) on spherical functions,
\begin{eqnarray*}
\frac{\Xi_{1}(g)}{\Xi_{t}(g)\Xi_{1-s-t}(g)}\leq
\frac{\Xi_1(g)^s}{c(G)^{2}}\leq \frac{\Xi_s(g)}{c(G)^3},
\end{eqnarray*}
yielding the announced inequality.
\end{proof}

So for $a_y,y<1$, from equation (\ref{comp})
\begin{eqnarray}
|\langle a_y v_\lambda, v_\lambda \rangle |\leq C|1-\log
y|y^{\frac{D}{2}-\lambda}
\end{eqnarray}
We will change the parameter for $\lambda$  from $0\leq \lambda \leq
\rho$ to $\frac{D}{2}\beta=\rho \leq \lambda \leq D\beta=2\rho$ so
that the bound becomes
$$C|1-\log y|y^{D-\lambda}$$
and when $\lambda=D$ it represents a trivial representation and when
$\lambda=D/2$ represents a principal series. Since $y^{\epsilon}\log
y\ra 0$ as $y\ra 0$ for any $\epsilon>0,$ we will write
\begin{eqnarray}\label{matrixcoefficient}
|\langle a_y v_\lambda, v_\lambda \rangle |\leq C
y^{D-\lambda-\epsilon } \end{eqnarray} for $\rho\leq \lambda\leq
2\rho$.
\section{Bottom eigenspectrum of Laplace operator}\label{laplace}
Henceforth, we use the notation $(V_\lambda,\pi_\lambda)$ to denote
the spherical principal or complementary unitary representation of
rank one group $G$. Let $X$ be a rank one symmetric space and fix a
origin $(0,0,1)=o\in X$. Let $M=\Gamma\backslash X$ be
geometrically finite with a critical exponent $\delta$.  Let
$B_\theta$ be a Busemann function based at $\theta$ normalised  that
$B_\theta(o)=0$. The function defined as
$$\phi_0(x,t,y)=\int_{\Lambda_\Gamma} e^{-\delta B_\theta(x,t,y)} d\theta$$ where $d\theta$ is a fixed
Patterson-Sullivan measure associated to a fixed reference point
$o\in X$, descends to a positive $L^2$ function on $M$ whose
eigenvalue with respect to the Laplace operator is
$\delta(D-\delta)$ where $D$ is the Hausdorff dimension of $\partial
X$. For $H^n_\bR,H^n_\C, H^n_\H$, $D=n-1,2n,4n+2$ resp. under the
normalization of the sectional curvature between $-4$ and $-1$. Note
that $L^2(\Gamma\backslash G)^K$ is naturally isomorphic to
$L^2(\Gamma\backslash H^n_\F)$ by averaging over $K$-orbits, and
this isomorphism intertwines the action of $C$, the Casimir
operator, with that of $-\Delta$. For geometrically finite groups,
the relation between the bottom spectrum $\lambda_0$ and the
critical exponent $\delta$(=Hausdorff dimension of the limit set) is
$$\lambda_0=\delta(D-\delta).$$

In this paper, we will assume that $\delta>D/2$ and
$(V_\lambda,\pi_\lambda), D/2\leq \lambda \leq D$ to denote the
spherical principal or complementary unitary representation of rank
one group $G$. Fix $D/2<s_\Gamma <\delta$  so that there is no
eigenvalue of the Laplace operator between $s_\Gamma(D-s_\Gamma)$
and the bottom spectrum $\delta(D-\delta)$ in $L^2(\Gamma\backslash
G/K)$. Such $s_\Gamma$ exists since the spectrum is discrete for
small eigenvalues, see \cite{H} for general geometrically finite
manifolds with pinched negative curvature. Then
$$L^2(\Gamma\backslash G)^K=V_\delta \oplus V$$ where $V_\delta$ is a
complementary series corresponding to $\delta$ and $V$ does not
contain any complementary series $V_s$ for $s\geq s_\Gamma$.

If $V_s$ is an irreducible factor for $V$, we want to estimate the
matrix coefficient decaying rate.

For $X_i$, an orthonormal basis of the Lie algebra of $K$ with
respect to an $Ad$-invariant scalar product, let $\omega=1-\sum
X_i^2$. This is a differential operator  in the center of the
enveloping algebra of Lie(K) and acts as a scalar on each
$K$-isotypic component of $V_s$.
\begin{proposition}\label{tempered}Let $(V,\pi)$
be a unitary representation of $\ G=SO(n,1)$, $SU(n,1), Sp(n,1)$
which do not weakly contain any complementary series representation
$V_s$ for $s \geq s_0$.  Then for any $\epsilon>0$, there exists
$c_\epsilon$ such that for any smooth vectors $w_1,w_2\in V$, and
$y<1$,
$$|\langle a_y w_1,w_2\rangle|\leq c_\epsilon
y^{D-s_0-\epsilon}||\omega^m(w_1)||||\omega^m(w_2)||,$$ where $4m>
rank(K)+2\#\{\text{positive\ roots}\}.$
\end{proposition}
\begin{proof}
$\pi$ decomposes as $\int_{\hat G} \oplus^{m_z}\rho_z d\nu$ where
$\hat G$ is the unitary dual of $G$, $m_z$ is the multiplicity of an
irreducible representation $\rho_z$, and $\nu$ is a spectral measure
on $\hat G$. Then $\rho_z$ is either tempered or isomorphic to a
complementary series $V_s$ for $D/2\leq s < s_0<n-1, D/2\leq s <
s_0< 2n, D/2\leq s < s_0< 4n$. Note that by Kostant, $s_0$ cannot
exceed $4n$ for quaternionic case, \cite{Ko}.

Then for $K$-finite unit vectors $w_1$ and $w_2$ of $\pi$, if we
write
$$w_i=\oplus v_i^\lambda,\ v_i^\lambda\in V_\lambda,$$ then by
Theorem \ref{tem}
$$\langle a_yv_1^\lambda, v_2^\lambda \rangle \leq
C(G)||v_1^\lambda||||v_2^\lambda||(\dim\Span(Kv_1^\lambda)\dim\Span(Kv_2^\lambda))^{1/2}
\Xi_{\lambda} (a_y).$$ By  equation (\ref{matrixcoefficient}), using
Cauchy-Schwarz inequality and $\dim\Span K w_i \geq
\dim\Span(Kv_i^\lambda)$, finally we get
$$ | \langle a_yw_1, w_2\rangle|\leq C(G)y^{D-s_0-\epsilon}\Pi \sqrt{\text{dim}\langle Kw_i\rangle}. $$

From $K$-finite vector to smooth vectors, it is standard, see
\cite{Mau}, that
$$ | \langle a_yw_1, w_2\rangle|\leq C(G)
y^{D-s_0-\epsilon}||\omega^m(w_1)||||\omega^m(w_2)||,$$ where
$4m>rank(K)+2\#\Sigma^+$ and $\Sigma^+$ is a set of positive roots
of $K$. In real hyperbolic $n$ space case ($K=SO(n)$),
$$m=0,1,2,4$$ for $n=2,3,4,6$ and
for $n\geq 5$ (odd), $m=(n-1)^2/4$ and for $n\geq 8$ (even),
$m=n^2/4$.
\end{proof}
From this we get
\begin{corollary}\label{product}
Let $\Gamma$ be a geometrically finite discrete subgroup of $G$ with
$\delta$ as in the standing assumption. Then for any $\phi_1\in
C^\infty(\Gamma\backslash G)^K\cap L^2(\Gamma\backslash G/M)$,
$\phi_2\in C_c^\infty(\Gamma\backslash G/M)$ and $0<y<1$,
$$\langle a_y\phi_1, \phi_2\rangle= \langle \phi_1,
\phi_0\rangle\langle a_y\phi_0, \phi_2\rangle+ O(y^{D-s_\Gamma}
S_{2m}(\phi_1)S_{2m}(\phi_2)).$$

\end{corollary}
\begin{proof}Note first that $$L^2(\Gamma\backslash G)^K=V_\delta \oplus V$$ where $V_\delta$ is a
complementary series corresponding to $\delta$ and $V$ does not
contain any complementary series $V_s$ for $s\geq s_\Gamma$. Put
$\phi_1=\langle \phi_1,\phi_0\rangle \phi_0+ \phi_1^\perp$. Since
$\phi_0$ is $K$-invariant, $\phi_1^\perp$ is also $K$-invariant.
Then
$$\langle a_y \phi_1,\phi_2 \rangle=\langle \phi_1,\phi_0\rangle
\langle a_y \phi_0, \phi_2\rangle+\langle a_y \phi_1^\perp, \phi_2
\rangle$$
$$=\langle \phi_1,\phi_0\rangle
\langle a_y \phi_0,
\phi_2\rangle+O(y^{D-s_\Gamma}S_{2m}(\phi_1)S_{2m}(\phi_2))$$ since
$S_{2m}(\phi_i^\perp)\ll S_{2m}(\phi_i)$ and by Proposition
\ref{tempered}.
\end{proof}


\section{Equidistribution}\label{equid}
Suppose the image of a horosphere is closed in $\Gamma\backslash
H^n_\F$ a geometrically finite manifold. Let $N(J)=\{[n]\in
(NM\cap\Gamma)\backslash NM:\Gamma\backslash \Gamma n A \cap J\neq
\emptyset\}$ where $J\subset \Gamma\backslash G$ a compact set,
i.e., the set of elements on $F_N$ whose orbit under the $A$-flow
intersects $J$. We use the notations of section \ref{pre} and
section \ref{bottom}.
\begin{lemma}$N(J)$ is bounded if $J$ is compact.
\end{lemma}
\begin{proof}Let $C_\Gamma$ be the convex hull of the limit set and $\Gamma\backslash C_\Gamma$ the convex core whose volume is finite due to
the geometrical finiteness of $\Gamma$. We claim that there exist
$L\gg 1$ and a fundamental domain $F$ of $\Gamma$ such that
\begin{eqnarray}\label{fundamentaldomain}
\{(z,t, v): (z,t)\in F_N, |(z,t)|> L\ \text{or}\ v >L\}\subset F.
\end{eqnarray}

If the base point $\infty$ of the horosphere is not in the limit set
of $\Gamma$, since the domain of discontinuity is open, there exists
a horoball based at $\infty$ which embeds into the manifold
$\Gamma\backslash H^n_\F$, hence embeds into a fundamental domain
$F$ of $\Gamma$.

If the base point is  a bounded parabolic fixed point and $F_N$ is
unbounded, since $\Gamma$ is geometrically finite, the intersection
$C_\Gamma \cap (F_N\times [0,\infty))\subset \{(z,t)\in F_N:|(z,t)|<
K\}\times [0,\infty)$ for some $K$. Otherwise since a horoball
$H=\{(z,t, v): v> R\}$ is stabilized by $\Gamma\cap NM$ for some
large $R$ and quotient down to a cusp of $\Gamma\backslash H^n_\F$,
the quotient of $H\cap C_\Gamma$ would have an infinite volume,
hence the volume of the convex core would be infinite. Take a
$\Gamma$-invariant small $\epsilon$-neighborhood $\cal N$ of
$C_\Gamma$. Then $ \cal N \cap (F_N\times [0,\infty))\subset
\{(z,t)\in F_N:|(z,t)|< K'\}\times [0,\infty)$ for some $K'$. The
nearest point retraction $r:(H^n_\F\setminus \cal N)\ra \partial
\cal N$ commutes with the action of $\Gamma$. Since $\partial \cal
N$ is invariant under $\Gamma$ and $F_N\times [0,\infty) \setminus
\cal N\subset r^{-1}(\partial \cal N \cap (F_N\times[0,\infty))$ due
to the convexity of $\cal N$, if we denote $\cal F$ a fundamental
domain of $\Gamma$ on $\partial \cal N$, and $F_C$ a fundamental
domain of $\Gamma$ on $C_\Gamma$,
$$r^{-1}(\cal F)\cup F_C=F,$$ and so for some large $L\gg 1$
$$\{(z,t)\in F_N: |(z,t)|> L\}\times [0,\infty) \subset F.$$
When $F_N$ is bounded, there is nothing to prove.

Now suppose $N(J)$ is unbounded. Then there exist sequences $n_j\in
F_N\ra \infty, a_j\in A, \gamma_j\in\Gamma$ and $w_j\in J$ such that
$n_ja_j=\gamma_j w_j$. As $J$ is bounded, we may assume that $w_j\ra
w\in G$.  Take $\gamma_0\in\Gamma$ such that $\gamma_0 w (0,0,1)\in
int F$ by choosing the base point different from $(0,0,1)$ if
necessary. Then for large $j$, $\gamma_0 w_j (0,0,1)\in int F$.
Since $n_j\ra\infty$, by (\ref{fundamentaldomain}),
$n_ja_j(0,0,1)\in F$ for large $j$. Therefore for large $j$
$$\gamma_0\gamma_j^{-1} n_j a_j (0,0,1)=\gamma_0w_j(0,0,1)\in int F\cap \gamma_0\gamma_j^{-1}(F). $$
Since $F$ is a fundamental domain of $\Gamma$, $\gamma_0=\gamma_j$
for all large $j$. Since $w_j\ra w$, $n_ja_j=\gamma_jw_j$ is
bounded, a contradiction. This shows that $N(J)$ is bounded.
\end{proof}

So for $\phi\in
C_c(\Gamma\backslash G/M)$ and any function $\eta\in
C_c((NM\cap\Gamma)\backslash NM/M)$ with $\eta|_{N(supp\phi)}=1$, we
have
$$\int_{(NM\cap\Gamma)\backslash NM/M}\phi(na_y)\eta(n)dn=\int_{N(supp\phi)}
\phi(na_y)dn$$ since $\phi(na_y)=0$ for $n$ outside  $N(supp\phi)$.
Hence we may assume that $\int_{(NM\cap\Gamma)\backslash NM/M}
\phi(na_y)dn=I_\eta(\phi)(a_y)$.

 Since $r_\epsilon$ is an $\epsilon$-approximation in
$A$-direction $S_{2m}(\rho_\epsilon)=O_\eta(\epsilon^{-p})$ for some
$p>0$, where $\rho_\epsilon=\rho_{\eta,\epsilon}$. Fix $l$ an
integer so that
$$l> \frac{(D-\delta)(p+1)}{(\delta-s_\Gamma)}.$$
Let $\psi_0=\phi$ and define
$$\psi_i=\psi'_{i-1},\  1\leq i \leq l$$ inductively.

 Then a similar technique in \cite{O} shows that
\begin{theorem}\label{equidistribution}For any $\phi\in C^\infty_c(\Gamma\backslash G)^K$,
$$\int_{F_N} \phi(ny)dn=\langle \phi,\phi_0 \rangle
c_{\phi_0}y^{D-\delta}(1+O(y^{\delta'})$$ for some constants
$c_{\phi_0}$ and $\delta'$. Here $\delta'$ depends only on the
spectral gap and the implied constant depends only on Sobolev norm
$S_m(\phi)$ and $\eta$, hence on $F_{\Lambda_\Gamma}$.
\end{theorem}
\begin{proof}Now for each $0\leq i \leq l-1$
$$I_\eta(\psi_i)(a_y)=\langle a_y\psi_i, \rho_\epsilon\rangle+
 O((\epsilon+y)I_\eta(\psi_{i+1})(a_y))$$ by Proposition \ref{another} and
 $$I_\eta(\psi_l(a_y))=\langle a_y\psi_l, \rho_\epsilon\rangle+
O((\epsilon+y)I_\eta(\psi_l')(a_y)).$$ Here since
$I_\eta(\psi_l')(a_y)=\int_{(NM\cap\Gamma)\backslash NM/M}
\psi_l'(na_y)\eta(n)dn\leq C \int_{(NM \cap\Gamma)\backslash NM/M}
\eta(n)dn$ where $C=\max\psi_l'(na_y)$ on $supp(\eta)$, we map write
$$I_\eta(\psi_l(a_y))=\langle a_y\psi_l, \rho_\epsilon\rangle+
O_\eta((\epsilon+y))$$ where the implied constant in $O_\eta$
depends on $\int_{(NM\cap\Gamma)\backslash NM/M} \eta(n)dn$.

By Corollary \ref{product}, for $1 \leq i \leq l$, using $S_k\psi_i
\ll S_{q}\psi_{i-1}$,
$$\langle a_y\psi_i, \rho_\epsilon \rangle=\langle \psi_i,
\phi_0\rangle \langle a_y\phi_0, \rho_\epsilon\rangle +
O(y^{D-s_\Gamma}S_{2m}(\psi_i)S_{2m}(\rho_\epsilon))$$
$$=O(\langle a_y\phi_0, \rho_\epsilon\rangle
||\psi_i||_2)+O(y^{D-s_\Gamma}\epsilon^{-p}S_{q}(\phi))$$
$$=O(\langle a_y\phi_0, \rho_\epsilon\rangle
S_{2m+l}(\phi))+O(y^{D-s_\Gamma}\epsilon^{-p}S_{q}(\phi)).$$

Hence for $y<\epsilon$,
$$I_\eta(\phi)(a_y)=\langle a_y\phi, \rho_\epsilon \rangle
+\sum_{k=1}^{l-1}O(\langle a_y\psi_k,
\rho_\epsilon\rangle(\epsilon+y)^k)+O_\eta(I_\eta(\psi_l)(a_y)(\epsilon+y)^l)$$
$$=\langle a_y\phi, \rho_\epsilon \rangle+O(\langle a_y\phi_0, \rho_\epsilon
\rangle \epsilon S_{q}(\phi))+ O(\epsilon
y^{D-s_\Gamma}\epsilon^{-p}S_{q}(\phi))+ O_\eta(\epsilon^l)$$
$$=\langle \phi, \phi_0\rangle \langle a_y\phi_0, \rho_\epsilon
\rangle + O(\langle a_y\phi_0, \rho_\epsilon \rangle \epsilon
)+O(y^{D-s_\Gamma}\epsilon^{-p})+O(\epsilon^l)$$
$$=\langle \phi, \phi_0\rangle \phi_0^N(a_y)+ O(y^\delta)+O(\epsilon
y^{D-\delta})+O(y^{D-s_\Gamma}\epsilon^{-p})+O(\epsilon^l)$$ by
Proposition \ref{basic}. All the implied constants depend on
$S_{q}(\phi)$ and $\int \eta dn$.

Setting $\epsilon=y^{(\delta-s_\Gamma)/(p+1)}$, we have
$\epsilon^l<y^{D-\delta}$ and
$$\int_{(NM\cap\Gamma)\backslash NM/M} \phi(na_y)dn=I_\eta(\phi)(a_y)=\langle
\phi,\phi_0 \rangle
\phi_0^N(a_y)(1+O(y^{(\delta-s_\Gamma)/(p+1)}))$$ using the fact
that $\phi_0^N(a_y)=c_{\phi_0}y^{D-\delta}+d_{\phi_0}y^\delta$   and
the fact $\delta>D/2> \frac{\delta-s_\Gamma}{p+1}+D-\delta$. The
claim follows by setting $\delta'=\frac{\delta-s_\Gamma}{p+1}$.
\end{proof}

  There exists a
Burger-Roblin measure $\hat\mu$ on $T^1X=G/M$
which descends to  $T^1(\Gamma\backslash X)$ and satisfies
$$\hat\mu(\psi)=\langle \psi, \phi_0\rangle $$ for $K$-invariant
functions $\psi\in C_c(G)$. Specially $\hat\mu(\phi_0)=1$. Indeed
one can write down $\hat\mu$ explicitly as follows. A fixed
Patterson-Sullivan measure $\nu_0$ on $\partial X=K/M$ can be
regarded as a measure on $K$ via the projection $K\ra K/M$: for
$f\in C(K)$,
$$\nu_0(f)=\int_{K/M}\int_M f(km)dmd\nu_0(k)$$ for the probability
invariant measure $dm$ on $M$. Using this, if $dg=y^{-D-1}dydn dk$
is a Haar measure on $G=NAK$, for $\psi\in C_c(G/M)$,
\begin{eqnarray}\label{Roblin-Burger}
\hat\mu(\psi)=\int_K\int_{AN} \psi(g) y^\delta y^{-D-1}dydnd\nu_0.
\end{eqnarray}
This measure is left $\Gamma$-invariant and descends to the measure
on $T^1(\Gamma\backslash X)=\Gamma\backslash G/M$.

Then
\begin{theorem}\label{equi}For $\psi\in C_c(T^1(\Gamma\backslash X))$,
$$\int_{n\in (NM\cap\Gamma)\backslash NM/M} \psi(na_y)dn\sim
c_{\phi_0}\hat\mu(\psi)y^{D-\delta}$$ as $y\ra 0$.
\end{theorem}
\begin{proof}See \cite{O}. The argument given there works for
general rank one space.
\end{proof}
\section{Orbital counting on a light cone}\label{counting}
In this section we give a quantitative estimate of the orbital
counting on $\F^{n+1}$ using the results in previous sections.

Let $\Gamma\subset G$ be a geometrically finite torsion-free
subgroup with $\delta>D/2$. Let $\langle v_0, v_0 \rangle=0$ where
$\langle\ , \rangle$ is the standard Hermitian form of signature
$(n,1)$ on $\F^{n,1}$. Let $L$ be the light cone of $\F^{n,1}$ given
by $\langle v, v \rangle=0$. Then $G$ acts transitively on $L$ and
so there is $g_0\in G$ such that the stabilizer of $v_0g_0$ is $NM$.
By conjugating $\Gamma$ by $g_0$, we may assume that the stabilizer
of $v_0$ in $G$ is $NM$.

 In
the language of section \ref{opposite}, then $v_0$ is proportional
to $(0,0,1)$ and the action of $a_y$ on it is
$$\left(\begin{matrix} 0 & 0 & 1\end{matrix}\right)\left[\begin{matrix}
 y & 0 & 0 \\
 0 & I & 0 \\
 0 & 0 & 1/y\end{matrix}\right] =\left(\begin{matrix}
                                                     0&
                                                     0&
                                                     \frac{1}{y}\end{matrix}\right),$$
i.e.,
$$ v_0 a_y=\frac{1}{y} v_0.$$ Traditionally we prefer $G$ to act on the right and $\Gamma$ on the left. So vectors in
$\F^{n,1}$ will be row vectors and the matrices act on the right. In
right action notation, $v_0 NM=v_0$, so the stabilizer of the row
vector $v_0$ is $NM$, and $v_0 a_y=\frac{1}{y}v_0$.
\begin{lemma}\label{closed}If $v_0\Gamma$ is discrete, then the image of the
horosphere $H_0$ corresponding to $v_0$ is closed in
$\Gamma\backslash H^n_\F$. Furthermore the image of the horosphere
$H_0$  is closed in $\Gamma\backslash H^n_\F$ if and only if $H_0$
is based either at a bounded parabolic fixed point or at the point
which is not in the limit set.
\end{lemma}
\begin{proof}Any horosphere is obtained by projectivizing $\Pi$
intersection with the hyperboloid where $\Pi$ is an affine
hyperplane $\Pi_X=\{Y\in \F^{n+1}: \langle\langle Z, X\rangle\rangle
=1\}$ with $X\in L\setminus 0$. Here $\langle\langle\ ,\
\rangle\rangle$ is a standard positive definite inner product of
signature $n+1$ on $\F^{n+1}$ and $X$ is uniquely determined up to
an element in $U(1)$ or $Sp(1)$. This gives a homeomorphism between
the set of horospheres and the set of $(L\setminus 0)/U(1)\ (resp.\
Sp(1))$.

Now suppose the orbit $v_0\Gamma$ is discrete, and $H_0$ to be the
horosphere corresponding to $v_0$. If $ \Gamma\backslash H_0$ is not
closed in $\Gamma\backslash H^n_\F$, there are points $p_i\in H_0$
so that $\pi(p_i)$ accumulates to a point $\bar p\notin
\Gamma\backslash H_0$ where $\pi:H^n_\F\ra \Gamma\backslash H^n_\F$.
This means that there exist $\gamma_i\in \Gamma$ so that $
p_i\gamma_i$ accumulates to $p\in H^n_\F$. Hence $ H_0\gamma_i$
accumulate, which is forbidden by discreteness of $ H_0\Gamma$. The
second claim is proved by Dal'Bo \cite{Dal} for the Hadamard
manifold whose sectional curvature is bounded above by $-1$. She
showed that the projection of the horosphere based at $\infty$ is
closed if and only if either $\infty$ is a bounded parabolic fixed
point or $\infty$ does not belong to the limit set
 provided that the length spectrum of $\Gamma$ is
non-discrete, which is shown in \cite{Kim} for rank one symmetric
space.
\end{proof}

 Define a
function on $\Gamma\backslash G$ by
 $$F_T(g)=\sum_{\gamma\in (\Gamma\cap NM)\backslash\Gamma} \chi_{B_T}(v_0 \gamma
 g),$$ where $B_T=\{v\in v_0G: ||v||<T\}$. Specially $F_T(e)$ is the
 orbital counting function of $\Gamma$.  Note that
for $\psi\in C_c(\Gamma\backslash G)$
$$\langle F_T, \psi\rangle=\int_{\Gamma\backslash G} \sum_{\gamma\in (\Gamma\cap NM)\backslash\Gamma} \chi_{B_T}(v_0 \gamma
 g) \psi(g) dg$$
$$=\int_{\Gamma\backslash G} [\sum_{\gamma\in (\Gamma\cap NM)\backslash\Gamma} \chi_{B_T}(v_0 \gamma g) \psi(\gamma g)] dg\ (
\psi\ \text{is}\ \Gamma\ \text{invariant})    $$
$$= \int_{(\Gamma\cap NM)\backslash G}  \chi_{B_T}(v_0  g) \psi(g) dg\ .
$$
Write $G$ as $G=(NM/M) A M (M\backslash K)$ and accordingly
$g=[n]a_y m k$. Then the above integral is

$$=\int_{M\backslash K} \int_{||v_0a_yk||<T} \int_{[n_x]\in (\Gamma\cap NM)\backslash NM/M} \psi([n_x]a_ymk)y^{-D-1}d[n]dmdydk   $$
$$=\int_{M\backslash K} \int_{y>T^{-1}||v_0k||}(\int_{[n_x]\in (\Gamma\cap NM)\backslash NM/M}\int_{m\in M} \psi([n_x]a_ymk)  dmd[n])
 y^{-D-1}dydk  $$
$$=\int_{M\backslash K} \int_{y>T^{-1}||v_0k||}\int_{[n_x]\in (\Gamma\cap NM)\backslash NM/M} \psi_k([n_x]a_y)d[n] y^{-D-1}dydk $$
\begin{eqnarray}\label{ave}
=\int_{M\backslash K} \int_{y>T^{-1}||v_0k||} \psi_k^N(a_y)
y^{-D-1}dydk,
\end{eqnarray} where $\psi_k(g)=\int_{m\in M} \psi(gmk)  dm$ and
$\psi^N(a_y)=\int_{(\Gamma\cap NM)\backslash NM/M} \psi([n]a_y)
d[n]$ as before.


\begin{lemma}\label{est}For any $\psi\in C_c(\Gamma\backslash G)$, as $T\ra \infty$
$$\langle F_T, \psi \rangle\sim
\delta^{-1}c_{\phi_0}T^\delta\int_{M\backslash K}
\hat\mu(\psi_k)||v_0k||^{-\delta}  dk.$$ If $\psi\in
C_c^\infty(\Gamma\backslash G)^K$ and $||\cdot||$ is $K$-invariant,
$$\langle F_T, \psi \rangle=\langle\psi, \phi_0\rangle
\delta^{-1}c_{\phi_0}T^\delta
||v_0||^{-\delta}(1+O(T^{-\delta'})).$$ Here the implied constants
depend only on Sobolev norm of $\psi$ and $F_{\Lambda_\Gamma}$ and
$supp(\psi)$.
\end{lemma}
\begin{proof}The function $\psi_k$ is $M$-invariant, hence it is defined on $T^1(\Gamma\backslash X)$, and by Theorem
\ref{equi},
$$ \psi^N_k(a_y)=\int_{n\in (\Gamma\cap NM)\backslash NM/M} \psi_k([n]a_y)d[n]\sim
c_{\phi_0}\hat\mu(\psi_k)y^{D-\delta}$$ as $y\ra 0$. By the equation
(\ref{ave})
$$\langle F_T, \psi \rangle\sim \int_{M\backslash K} \int_{y>T^{-1}||v_0k||} c_{\phi_0}\hat\mu(\psi_k)y^{D-\delta}  y^{-D-1}dydk$$
$$=\delta^{-1}c_{\phi_0}T^\delta\int_{M\backslash K}
\hat\mu(\psi_k)||v_0k||^{-\delta}  dk.$$  If $\psi$ and $||\cdot||$
are $K$-invariant, since $\psi_k=\psi$, by Theorem
\ref{equidistribution},
$$\psi^N(a_y)=\langle \psi,\phi_0 \rangle
c_{\phi_0}y^{D-\delta}(1+O(y^{\delta'})).$$
Hence $$\langle F_T, \psi \rangle=\int_{M\backslash K}
\int_{y>T^{-1}||v_0k||} \langle \psi,\phi_0 \rangle
c_{\phi_0}y^{D-\delta}(1+O(y^{\delta'})) y^{-D-1}dydk$$
$$=dk(M \backslash K)\int_{y>T^{-1}||v_0||} \langle \psi,\phi_0
\rangle c_{\phi_0}y^{D-\delta}(1+O(y^{\delta'}))    y^{-D-1}dy$$
$$=dk(M\backslash K)\langle \psi,\phi_0 \rangle
c_{\phi_0}\delta^{-1}T^\delta||v_0||^{-\delta}(1+O(T^{-\delta'})).$$
\end{proof}
\begin{theorem}As $T\ra \infty$,
$$F_T(e)\sim c_{\phi_0}\delta^{-1}T^\delta \int_{ K} ||v_0k||^{-\delta}dk.$$ If $||\cdot||$ is
$K$-invariant,
$$F_T(e)=\delta^{-1}\phi_0(e)c_{\phi_0}||v_0||^{-\delta}T^\delta(1+O(T^{-\delta'})).$$
\end{theorem}
\begin{proof}
For small $\epsilon$, choose a symmetric $\epsilon$-neighborhood
$U_\epsilon$ of $e$ in $G$, which injects to $\Gamma\backslash G$,
such that for all $T\gg 1$ and $0<\epsilon\ll 1$,
$$B_TU_\epsilon\subset B_{(1+\epsilon)T}\ \text{and}\
B_{(1-\epsilon)T}\subset \cap_{u\in U_\epsilon} B_Tu.$$ Let
$\phi_\epsilon\in C_c^\infty(G)$ be a nonnegative function supported
on $U_\epsilon$ with $\int_G \phi_\epsilon=1$. Define
$\Phi_\epsilon$  on $\Gamma\backslash G$ by averaging over $\Gamma$
$$\Phi_\epsilon(\Gamma g)=\sum_{\gamma\in
\Gamma}\phi_\epsilon(\gamma g).$$ By  definition,
$\Phi_\epsilon([g])\neq 0$ only if $[g]\in \Gamma\backslash \Gamma
U_\epsilon$. Also for $g\in U_\epsilon$, if $||v_0\gamma||<T$ then
$||v_0\gamma g||<(1+\epsilon)T$.

So $$\langle F_{(1+\epsilon)T}, \Phi_\epsilon \rangle$$
$$=\int_{\Gamma\backslash G} \sum_{\gamma\in (\Gamma\cap NM)\backslash\Gamma}\chi_{B_{(1+\epsilon)T}}(v_0\gamma g)\Phi_\epsilon(g)
dg=\int_{U_\epsilon} \sum_{\gamma\in (\Gamma\cap
NM)\backslash\Gamma}\chi_{B_{(1+\epsilon)T}}(v_0\gamma g)
\phi_\epsilon(g) dg$$
$$=\sum_{\gamma\in (\Gamma\cap NM)\backslash\Gamma}\int_{U_\epsilon}\chi_{B_{(1+\epsilon)T}}(v_0\gamma g)
\phi_\epsilon(g) dg.$$ Note that
$\int_{U_\epsilon}\chi_{B_{(1+\epsilon)T}}(v_0\gamma g)
\phi_\epsilon(g) dg\leq \int_{U_\epsilon} \phi_\epsilon(g) dg=1$ and
the equality holds only if $||v_0\gamma
Supp(\phi_\epsilon)||<(1+\epsilon)T$. But  if $||v_0\gamma||<T$,
then $||v_0\gamma U_\epsilon||<(1+\epsilon)T$. Hence
$$\langle F_{(1+\epsilon)T}, \Phi_\epsilon \rangle \geq F_T(e).$$
Similarly
$$\langle F_{(1-\epsilon)T}, \Phi_\epsilon \rangle$$
$$=\sum_{\gamma\in (\Gamma\cap NM)\backslash\Gamma}\int_{U_\epsilon}\chi_{B_{(1-\epsilon)T}}(v_0\gamma g)
\phi_\epsilon(g) dg.$$ Since $||v_0\gamma
Supp(\phi_\epsilon)||<(1-\epsilon)T$ if $||v_0\gamma||<T$, $\langle
F_{(1-\epsilon)T},\Phi_\epsilon \rangle\leq F_T(e)$.

 So we obtain
 \begin{eqnarray}\label{squeeze}
\langle F_{(1-\epsilon)T},\Phi_\epsilon \rangle\leq F_T(e)\leq
\langle F_{(1+\epsilon)T}, \Phi_\epsilon \rangle.
\end{eqnarray}
 By Lemma
\ref{est},
$$\langle F_{(1\pm \epsilon)T}, \Phi_\epsilon\rangle\sim
\delta^{-1}c_{\phi_0}(T(1\pm\epsilon))^\delta\int_{M \backslash K}
\hat\mu((\Phi_\epsilon)_k)||v_0k||^{-\delta} dk.$$ But
$$\hat\mu((\Phi_\epsilon)_k)=\int_{\Gamma\backslash G}\int_M
\Phi_\epsilon(gmk)dmd\hat\mu(g).$$ Hence
$$\langle F_{(1\pm \epsilon)T}, \Phi_\epsilon\rangle\sim
\delta^{-1}c_{\phi_0}(T(1\pm\epsilon))^\delta\int_{M\backslash
K}\int_{g\in \Gamma\backslash G}\int_M
\Phi_\epsilon(gmk)dmd\hat\mu(g) ||v_0k||^{-\delta} dk$$
$$    =    \delta^{-1}c_{\phi_0}(T(1\pm\epsilon))^\delta\int_{\Gamma\backslash G}\int_{ K} \Phi_\epsilon(gk) ||v_0k||^{-\delta} dk
d\hat\mu(g).$$ For a fixed $k_0$, using Equation
(\ref{Roblin-Burger}) $$\int_{\Gamma\backslash G}\Phi_\epsilon(gk_0)
d\hat\mu(g)=\int_{\Gamma\backslash G}\Phi_\epsilon(gk_0) y^\delta
\frac{dg}{dk} d\nu_0=\int_{\Gamma\backslash
G}\phi_\epsilon(g)y^\delta \frac{dg}{dk}(k_0^{-1})^* d\nu_0.$$ But
if $\epsilon$ is small, $y$ runs, say from $1-\epsilon$ to
$1+\epsilon$. Since $\int \phi_\epsilon dg=1$,
$\int_{\Gamma\backslash G}\phi_\epsilon(g)y^\delta \frac{dg}{dk}
(k_0^{-1})^* d\nu_0$ is in the order of
$\int_{1-\epsilon}^{1+\epsilon} y^\delta=O(\epsilon)$.

Hence the above equality becomes
$$\langle F_{(1\pm \epsilon)T}, \Phi_\epsilon\rangle\sim     \delta^{-1}c_{\phi_0}(T(1\pm\epsilon))^\delta( \int_{ K}  ||v_0k||^{-\delta} dk + O(\epsilon)).     $$
In equation (\ref{squeeze}), let $\epsilon\ra 0$
to obtain
$$F_T(e)\sim \delta^{-1}c_{\phi_0}T^\delta \int_{ K} ||v_0k||^{-\delta}.$$

If $||\cdot||$ is $K$-invariant, take $U_\epsilon$ and
$\phi_\epsilon$ $K$-invariant, and we obtain
$$\langle F_{(1\pm \epsilon)T}, \Phi_\epsilon\rangle=\langle \Phi_\epsilon,\phi_0 \rangle
c_{\phi_0}\delta^{-1}(T(1\pm\epsilon))^\delta||v_0||^{-\delta}(1+O(T^{-\delta'})).$$
Since $U_\epsilon$ must contain $K$ and $dm$ is a probability
measure on $M\subset K$, we can take $\phi_\epsilon(e)=1$ to have
$\int \phi_\epsilon dg=1$. Hence $\langle \Phi_\epsilon,\phi_0
\rangle=\int_{g\in U_\epsilon} \phi_\epsilon(g)\phi_0(g)
dg=\phi_0(e)+R_\epsilon$ and $R_\epsilon\ra 0$ as $\epsilon\ra 0$.
So we obtain
$$F_T(e)=\phi_0(e)c_{\phi_0}\delta^{-1}(T)^\delta||v_0||^{-\delta}(1+O(T^{-\delta'})).$$
\end{proof}

Finally the orbital counting theorem can be stated as follows:
\begin{theorem}Let $\Gamma\subset G=KAN$ be a geometrically finite
group with critical exponent $\delta>\frac{D}{2}$ for real, complex
and quaternionic hyperbolic space, where $KAN$ is a fixed Iwasawa
decomposition introduced in section \ref{pre}.
Suppose $v_0$ is in the light cone such that $v_0\Gamma$ is
discrete, and the stabilizer of $v_0$ in $g_0^{-1}\Gamma g_0$ is in
$NM$.  Then for any norm $||\cdot||$ on $\F^{n,1}$
$$\#\{v\in v_0\Gamma:||v||<T\}\sim c_{\phi_0}\delta^{-1}T^\delta\int_{ K} ||v_0(g_0^{-1}kg_0)||^{-\delta}dk.$$  If $||\cdot||$ is
$g_0^{-1}Kg_0$-invariant, then
$$\#\{v\in
v_0\Gamma:||v||<T\}=c_{\phi_0}\delta^{-1}T^\delta||v_0||^{-\delta}(1+O(T^{-\delta'})).$$
\end{theorem}

\section{application to Apollonian packing}\label{application}
In euclidean $n$-space, the maximum number of mutually tangent
$n-1$-spheres is $n+2$, if one demands that no three spheres have a
common point. An Apollonian sphere packing of $n$-dimensional
euclidean space is for a given $n+1$ mutually tangent spheres, one
adds a  sphere which is tangent to all the $n+1$ previous spheres.
This procedure defines a sphere packing, called Apollonian sphere
packing.

The formula for curvatures for such an $n+2$ mutually tangent
spheres is  \cite{cox, BD}
$$n\sum_{i=1}^{n+2}\kappa_i^2=(\sum_{i=1}^{n+2}\kappa_i)^2.$$
For a given $n+1$ mutually tangent spheres, there are exactly two
spheres satisfying the above equation. The sum of curvatures of
these two last spheres is
$$\kappa_{n+2}+\kappa_{n+2}'=\frac{2}{n-1}\sum_{i=1}^{n+1}\kappa_i.$$
Hence for $n\leq 3$, if the first $n+2$ generating spheres have
integral curvatures, then the curvature $\kappa'_{n+2}$ of the next
sphere is also integral.  So we restrict our attention to integral
Apollonian packings.

Let
$$Q(\kappa_1,\cdots,\kappa_{n+2})=n\sum_{i=1}^{n+2}\kappa_i^2-(\sum_{i=1}^{n+2}\kappa_i)^2$$
be a quadratic form on $\bR^{n+2}$. This has signature $(n+1,1)$.
Hence the orthogonal group $O_Q$ can be identified with the group of
hyperbolic isometry in $H^{n+1}_\bR$, $O(n+1,1)$.

From now on we focus on $n=3$, hence
$$\kappa_{5}+\kappa_{5}'=\sum_{i=1}^{4}\kappa_i.$$

There is an integral group, called Apollonian group, defined as
follows. Changing the curvature
$(\kappa_1,\cdots,\kappa_4,\kappa_5)$ to
$(\kappa_1,\cdots,\kappa_4,\kappa_5')$ corresponds to the integral
matrix
\[S_5=\left(\begin{matrix}
   1 & 0 & 0 & 0 & 0 \\
   0 & 1 & 0 & 0 & 0 \\
   0 & 0 & 1 & 0 & 0 \\
   0 & 0 & 0 & 1 & 0 \\
   1 & 1 & 1 & 1 & -1 \end{matrix}\right)\]
Similarly there exist corresponding matrices $S_1,S_2,S_3$ and
$S_4$. Denote $\cal A=\langle S_1,S_2,S_3,S_4,S_5 \rangle\subset
O_Q(\mathbb Z)$, the Apollonian group. This group acts on apollonius
sphere packing $\cal P$ generated by the initial spheres with
curvature $\kappa_1,\cdots,\kappa_5$. We will see shortly that this
Kleinian group is geometrically finite.

If $S^1,\cdots,S^5$ are initial spheres, then next sphere ${S^5}'$
is obtained by a reflection along the sphere $T^5$ which passes
through the intersection points of $S^1,S^2,S^3,S^4$.  So the
element $S_5$ corresponds to a reflection $R_5$ along $T^5$. It is
similar for other elements $S_i$. This way $\cal A$ corresponds to a
group $\Gamma$ generated by reflections $R_i$.

Let $D_i$ be a hemisphere in upper half space of $\bR^4$, whose
boundary is $T^i$. Since $\bR^3$ is the ideal boundary of $H^4_\bR$
and the reflections in $\bR^3$ generates whole isometry group of
$H^4_\bR$, we know that $\Gamma$ is a discrete subgroup of
$Iso(H^4_\bR)$ generated by reflections along $D_i$  with a
fundamental domain bounded by $D_i$. Since the fundamental domain
has finite sides, it is a geometrically finite group with cusps.

Note that $\Gamma$ acts on $\cal P$ by permuting spheres in $\cal
P$, so it leaves invariant the set of tangent points between
spheres. But it is easy to see that any point on some sphere can be
approximated by tangent points, hence the limit set of $\Gamma$ is
equal to the union of spheres in $\cal P$. Since a sphere has
Hausdorff dimension $2$, the Hausdorff dimension of the limit set
$\Lambda_\Gamma>2>3/2$. Hence our assumption that the critical
exponent is greater than $D/2$ is satisfied to apply previous
results to this section.

\subsection{Orbital counting for $\cal A$}
We use a $K$-invariant norm $||\ ||$ in this section, in fact a
maximum norm,  and $G=SO^0(4,1)$. As before
$B_T=\{v\in\bR^5:Q(v)=0,\ ||v||<T\}$, i.e., the set of curvatures of
maximal mutually tangent spheres in $\bR^3$ whose norm is less than
$T$. Fix an initial curvatures $\xi$ so that $Q(\xi)=0$ and the
corresponding spheres generate an Apollonian sphere packing in
$\bR^3$. By conjugation if necessary, we will assume that the
stabilizer of $\xi$ is equal to $NM$. One has to recall the
definitions and notations from section \ref{counting}. Let
$\phi_\epsilon\in C^\infty_c(H^4_\bR)=C^\infty_c(SO^0(4,1))^K$ be  a
nonnegative function supported on a $K$-invariant $U_\epsilon$ with
$\int_{H^4_\bR} \phi_\epsilon=1$. Define a function defined on
$M=\Gamma\backslash H^4_\bR=\Gamma\backslash G/K$ by
$$\Phi_\epsilon(\Gamma g)=\sum_{\gamma\in
\Gamma}\phi_\epsilon(\gamma g).$$   Fix a primitive integral
polynomial $f$ in 5 variables. One defines a positive sequence
$A(T)=\{a_n(T)\}$ where
$$a_n(T)=\sum_{\gamma\in{\text {Stab}}_\Gamma(\xi)\backslash \Gamma, f(\xi\gamma)=n}w_T(\gamma)$$ where
$w_T(\gamma)=\int_{G/K}\chi_{B_T}(\xi\gamma
g)\phi_\epsilon(g)d\mu(g)$ for each $\gamma\in\Gamma$. A subsequence
$A_d(T)=\{a_n(T)|n=0(\text{mod}\ d)\}$ and define $|A(T)|=\sum
a_n(T)$, $|A_d(T)|=\sum_{n=0(\text{mod}\ d)}a_n(T)$.

For any subgroup $\Gamma_0\subset \Gamma$ with $\text
{Stab}_\Gamma(\xi)=\text{Stab}_{\Gamma_0}(\xi)$, define
$$F^{\Gamma_0}_T(g)=\sum_{\gamma\in{\text {Stab}}_\Gamma(\xi)\backslash \Gamma_0}\chi_{B_T}(\xi\gamma g),$$ and
$$\Phi^{\Gamma_0}_{\epsilon,\gamma_1}(g)=\sum_{\gamma\in\Gamma_0}\phi_\epsilon(\gamma_1^{-1}\gamma
g),\ \gamma_1\in\Gamma,$$ which is an $\epsilon$-approximation to
the identity around $[\gamma_1^{-1}]$ in $\Gamma_0\backslash
H^4_\bR$. Finally let
$$\Gamma_\xi(d)=\{\gamma\in\Gamma: \xi\gamma=\xi(\text{mod}\ d)\},$$
and its finite index subgroup
$$\Gamma(d)=\{\gamma\in\Gamma: \gamma=I(\text{mod}\ d)\}.$$

Then $\text{Stab}_{\Gamma}(\xi)=\text{Stab}_{\Gamma_\xi(d)}(\xi)$.
The following is due to Bourgain-Gamburd-Sarnak \cite{s}, Salehi
Golsefidy-Varj\'u \cite{VS}, Breuillard-Green-Tao \cite{BGT} and
Pyber-Szabo \cite{PS}.
\begin{theorem}For any square free integer $d$, the $L^2$-spectrum of
$\Gamma(d)\backslash H^4_\bR$ has a uniform spectral gap $2\leq
\theta<\delta_\Gamma$ so that it does not have any spectrum in
$[\theta(3-\theta),\delta_\Gamma(3-\delta_\Gamma))$.
\end{theorem}
\begin{proof}({\bf Step I; Family of expanders}) For general Apollonian sphere packing in
$\bR^n$, $$\Gamma=\cal A=\langle S_1,\cdots,S_{n+2} \rangle\subset
G=O_Q(\mathbb Z[\frac{1}{n-1}])\subset GL(n+2,\mathbb
Z[\frac{1}{n-1}]).$$ Note that $\Gamma$ is not discrete if $n\geq
4$. Then the following is proved by A. Salehi Golsefidy and P.
Varj\'u \cite{VS}.
\begin{thA} Let $\Gamma \subset GL(d, (\mathbb Z[\frac{1}{q_0}])$ be
a group generated by a symmetric set $S$. Then
$Cayley(\Gamma/\Gamma(q);\ S\ mod\ q)$ form a family of expanders
when $q$ ranges over square-free integers coprime to $q_0$ if and
only if the connected component of the Zariski-closure of $\Gamma$
is perfect.
\end{thA}

({\bf Step II; Uniform spectral gap}) This is a corollary of the
above theorem. For reader's convenience, following closely the
argument in \cite{s}, we give the details of its extension to higher
dimensional real hyperbolic spaces. Suppose $\Gamma$ is a
geometrically finite discrete group in $G=O(n,1)$. For large enough
$q$, by \cite{MVW} $\Gamma/\Gamma(q)=G(\F_q)$  where $\F_p=\Z/p\Z$.

Fix a fundamental domain $\cal F\subset H^{n+1}_\br$ of $\Gamma$.
For any map $f$ defined on $\Gamma(q)\backslash H^{n+1}_\br$, let
$\tilde f$ be a lift to $H^{n+1}_\br$. Then $f$ can be regarded as a
vector valued function $F$ defined on $\Gamma\backslash H^{n+1}_\br$
by setting
$$F(z)=(\tilde f (\gamma z))_{\gamma\in \Gamma/\Gamma(q)}$$ where $z\in \cal F$ by identifying $\Gamma\backslash H^{n+1}_\br$
with $\cal F$. It satisfies the equivariant condition
$$F(\gamma z)=R(\gamma)F(z)$$ for $\gamma\in G(\F_q)$, where
$R(\gamma)$ denotes the right regular representation of $G(\F_q)$.

More formally, every $L^2$ function on $\Gamma(q)\backslash
H^{n+1}_\br$ can be viewed as an $L^2$ section of a flat bundle $E$
of rank $|G(\F_q)|$ on $\Gamma\backslash H^{n+1}_\br$ as follows:
 Given $f\in
L^2(\Gamma(q)\backslash H^{n+1}_\br,\br)$, define $F:H^{n+1}_\br\ra
L^2(G(\F_q))$ by $$F(x,g)=\tilde f(gx)$$ for  $ x\in\cal F,
g\in\Gamma$. Then, for $h\in \Gamma$, $$F(hx,g)=\tilde f(ghx)=(R(h\
\text {mod} \Gamma(q))F)(x,g),$$ so $F$ can be viewed as a section
of the flat vector bundle $E$ over $\Gamma\backslash H^{n+1}_\br$
associated to the $\Gamma$-principal bundle $H^{n+1}_\br \ra
\Gamma\backslash H^{n+1}_\br$ via the right regular representation
of $G(\F_q)$ composed with $\Gamma \ra \Gamma/\Gamma(q)=G(\F_q)$.
Or, equivalently, $E$ is associated to the $G(\F_q)$-principal
bundle $\Gamma(q)\backslash H^{n+1}_\br \ra \Gamma\backslash
H^{n+1}_\br$ via the right regular representation $R$ of $G(\F_q)$.

Then $E=E_0 + E_1$ according to the splitting of the right regular
representation into constant functions, and functions on $G(\F_q)$
which are orthogonal to constant functions. This splitting is
orthogonal and parallel (compatible with the flat connection). The
splitting of $F=F_0 +F_1$ corresponds to $f=f_0 +f_1$ where $f_0$ is
$G(\F_q)$-invariant.

If $f$ is orthogonal to the $\lambda_0$ eigenspace of
$\Gamma(q)\backslash H^{n+1}_\br$, which is generated by a
$G(\F_q)$-invariant function, then each factor of $f$ is orthogonal
to the $\lambda_0$ eigenspace of $\Gamma\backslash H^{n+1}_\br$.
Integrating each factor over $\Gamma\backslash H^{n+1}_\br$ and
taking summation, we get $|| \nabla f ||_2^2 \geq \lambda_1
(\Gamma\backslash H^{n+1}_\br) || f ||_2^2$. Here $||\cdot||_2$
denotes the $L^2$-norm over $\Gamma(q)\backslash H^{n+1}_\br$.

Hence from now on, we can assume that $f=f_1$, $F=F_1$ takes values
in $E_1$.

 Pick a generating system $S$ of $\Gamma$. Construct corresponding right
Cayley graph. By definition of the spectral gap for the discrete
Laplacian, for each $z\in \Gamma\backslash H^{n+1}_\br$,
$$\frac{1}{|S|} \sum_{s\in S}||F(z) -R(s)F(z)||^2 \geq
\lambda_1(G(\F_q))||F(z)||^2.$$ Integrate over $\Gamma\backslash
H^{n+1}_\br$ and get
$$\frac{1}{|S|} \sum_{s\in S}|| F-R(s)F ||_2^2 \geq
\lambda_1(G(\F_q))|| F ||_2^2.$$

Let $H_0$ be the subspace of such functions orthogonal to the bottom
eigenfunction $\phi_0$. We need to show that there exists
$\epsilon>0$ independent of $q$ such that for any $F$ in $H_0$
$$\frac{\int_\cal F ||\nabla F||^2 d\mu}{\int_\cal F||F||^2d\mu}\geq \lambda_0 +\epsilon.$$
Above discussion about expanders implies that for $z\in\cal F$, and
for any $F\in H_0$, there exists $\gamma\in S$ such that
\begin{eqnarray}\label{eq}
||F(\gamma z)-F(z)||^2\geq \lambda_1(G(\F_q)) ||F(z)||^2,\  \text {or}\\
||R(\gamma)F-F||_2^2\geq \lambda_1(G(\F_q)) ||F||_2^2.\end{eqnarray}
Set  $f=||F||=a\phi_0(z)+b(z)$ where $b$ is orthogonal to $\phi_0$,
the $\lambda_0$ eigenfunction normalized that $\int_\cal
F\phi_0^2=1$, and $L^2$ norm of $F$, $a^2+\int_\cal F b(z)^2d\mu$ is
1. Then one can show that
$$\frac{\int_\cal F ||\nabla F||^2 d\mu}{\int_\cal F||F||^2d\mu}\geq
\lambda_0+(\lambda_1-\lambda_0)\int_\cal F b^2 d\mu.$$ If $\int_\cal
F b^2 d\mu$ is bounded below for all $F\in H_0$ we are done. Hence
suppose there is no lower bound. Then after taking a weak limit, we
may assume that there is $F\in H_0$ with $\int_\cal F b^2 d\mu=0$
and $a=1$. One can write
$$F=(F_i),\  \frac{F_i}{\phi_0}=u_i$$ so that $\sum u_i^2=1$, which implies that $\sum u_j\frac{\partial u_j}{\partial x_i}=0$.
Set $u=(u_i)$. Then Equation (\ref{eq}) becomes
\begin{eqnarray}\label{eq1}
 ||R(\gamma)u(z)-u(z)||^2 > \lambda_1(G(\F_q))
\end{eqnarray} for any $z\in \cal F$ and  some $\gamma\in S$. By
Step I, $\lambda_1(G(\F_q))$ is uniformly bounded below independent
of $q$. A direct calculation as in \cite{s} shows that
$$||\nabla(\phi_0u)||^2= \sum_j|\nabla(\phi_0 u_j)|^2$$
$$=\sum_{j,i}(\frac{\partial(\phi_0u_j)}{\partial x_i})^2=
\sum_{i,j}(u_j\frac{\partial \phi_0}{\partial
x_i}+\phi_0\frac{\partial u_j}{\partial x_i})^2=\phi_0^2||\nabla
u||^2+||\nabla \phi_0||^2,$$ which implies that
$$\frac{\int_\cal F ||\nabla F||^2 d\mu}{\int_\cal F||F||^2d\mu}=\frac{\int_\cal F ||\nabla \phi_0||^2+\phi_0^2||\nabla u||^2d\mu}{\int_\cal F |\phi_0|^2d\mu}\geq
\lambda_0+\frac{\int_\cal F \phi_0^2||\nabla u||^2d\mu}{\int_\cal F
\phi_0^2d\mu}.$$  If $\frac{\int_\cal F \phi_0^2||\nabla
u||^2d\mu}{\int_\cal F \phi_0^2d\mu}$ is bounded below, we are done
again. Remember that we normalize $\phi_0$ so that $\int_\cal F
\phi_0^2 =1$.

Now $u=(u_1,\cdots,u_k)$ for $k=|G(\F_q)|$ and hence $\nabla
u=(\nabla u_1,\cdots,\nabla u_k)$. Fix $z_0\in\cal F$. Take a unit
speed geodesic $\alpha(t)$ connecting $z_0$ and $\gamma z_0$. Then
$$ u(\gamma z_0)-u(z_0)=\int_\alpha \frac{du(\alpha(t))}{dt}
dt=\int_\alpha \nabla u(\alpha(t)) \alpha'(t)dt.$$ Hence
$$ ||u(\gamma z_0)-u(z_0)||^2=(\int_\alpha \frac{du(\alpha(t))}{dt}
dt)^2\leq d(z_0,\gamma z_0)\int_\alpha ||\nabla u( \alpha(t))\cdot
\alpha'(t)||^2 dt$$$$\leq d(z_0,\gamma z_0)\int_\alpha ||\nabla
u(\alpha(t))||^2 dt
$$
Take a Fermi coordinate along $\alpha$ so that metric tensor is
$$ ds^2= d\rho^2 + \sinh^2\rho d\phi^2+ \cosh^2\rho dt^2$$ where $\rho$ is a
distance from $\alpha$ so that the volume form
$d\mu=\cosh\rho\sinh^{n-1}\rho dt d\rho d\phi$. Take a small cross
section $B$ orthogonal to $\alpha$ passing through $z_0$ so that the
induced volume form on $B$ can be written as $dB=\sinh^{n-1}\rho
d\rho d\phi$. Then $N=B\times \alpha$ will be a small tubular
neighborhood around $\alpha$ so that
$$L=\max \{d(z,\gamma z): z\in B\}<\infty,\ m=\min_N \phi_0^2>0.$$
Since $\phi_0$ is a $\Gamma$-invariant positive function, $m>0$, see
section \ref{laplace}. By integrating above inequality over $B$,
$$\int_B ||u(\gamma z)-u(z)||^2 dB
\leq L \int_N ||\nabla u||^2 dtdB \leq L \int_N ||\nabla
u||^2\cosh\rho dtdB$$$$= L \int_N ||\nabla u||^2 d\mu\leq
\frac{L}{m}\int_N \phi_0^2||\nabla u||^2 d\mu\leq
\frac{L}{m}\int_\cal F \phi_0^2||\nabla u||^2 d\mu.$$ By equation
(\ref{eq1}), $\int_\cal F \phi_0^2||\nabla u||^2 d\mu$ has a uniform
lower bound.
\end{proof}

Since $\Gamma(d)$ is a finite index subgroup of $\Gamma_\xi(d)$,
such a spectral gap theorem holds for $\Gamma_\xi(d)$ as well. From
this and Lemma \ref{est} we obtain
\begin{proposition}\label{uniform}There exists $\epsilon_0=\delta'$ uniform over all square
free integer $d$ so that for any $\gamma_1\in\Gamma$ and for any
$\Gamma_\xi(d)$, we have
$$\langle F^{\Gamma_\xi(d)}_T,
\Phi^{\Gamma_\xi(d)}_{\epsilon,\gamma_1}\rangle_{L^2(\Gamma_\xi(d)\backslash
H^4_\bR)}=\frac{c_{\phi_0}d_\epsilon}{\delta_\Gamma
[\Gamma:\Gamma_\xi(d)]}||\xi||^{-\delta_\Gamma}T^{\delta_\Gamma}(1+O(T^{-\epsilon_0})).$$
\end{proposition}
\begin{proof}By Lemma \ref{est},
for $\Phi^{\Gamma_\xi(d)}_{\epsilon,\gamma_1}\in
C_c^\infty(\Gamma_\xi(d)\backslash G)^K$ and $K$-invariant
$||\cdot||$, since $\Gamma_\xi(d)$ is finite index in $\Gamma$,
$\delta_{\Gamma_\xi(d)}=\delta_\Gamma$ and
$$\langle F_T^{\Gamma_\xi(d)}, \Phi^{\Gamma_\xi(d)}_{\epsilon,\gamma_1} \rangle=
\langle\Phi^{\Gamma_\xi(d)}_{\epsilon,\gamma_1}, \phi_0^{\Gamma_\xi(d)}\rangle
{\delta_\Gamma}^{-1}c_{\phi_0^{\Gamma_\xi(d)}}T^{\delta_\Gamma}
||\xi||^{-\delta_\Gamma}(1+O(T^{-\delta'})).$$  Here since $\delta'$
depends only on the spectral gap and since it is uniform over all
square free integers $d$, it is uniform. Let $\tilde \phi_0$ be a
lift to $M(d)=\Gamma_\xi(d)\backslash H^4_\bR$ of the bottom unit
eigenfunction $\phi_0$ on $M=\Gamma\backslash H^4_\bR$. Then the
unit bottom eigenfunction on $M(d)=\Gamma_\xi(d)\backslash H^4_\bR$
is
$$\phi_0^{\Gamma_\xi(d)}=\frac{1}{\sqrt{[\Gamma:\Gamma_\xi(d)]}}\tilde\phi_0.$$
Since $\Phi^{\Gamma_\xi(d)}_{\epsilon,\gamma_1}$ is an
$\epsilon$-approximation to the identity around $[\gamma^{-1}]$ in
$M(d)$ and $\tilde\phi_0$ is invariant under $\Gamma$,
$$\langle\Phi^{\Gamma_\xi(d)}_{\epsilon,\gamma_1},
\phi_0^{\Gamma_\xi(d)}\rangle=\langle\Phi^{\Gamma_\xi(d)}_{\epsilon,e},\frac{1}{\sqrt{[\Gamma:\Gamma_\xi(d)]}}\tilde
\phi_0\rangle_{L^2(M(d))}$$$$=\frac{1}{\sqrt{[\Gamma:\Gamma_\xi(d)]}}\langle
\Phi_\epsilon, \phi_0
\rangle_{L^2(M)}=\frac{1}{\sqrt{[\Gamma:\Gamma_\xi(d)]}}d_\epsilon.$$
Also note that
$c_{\phi_0^{\Gamma_\xi(d)}}=\frac{1}{\sqrt{[\Gamma:\Gamma_\xi(d)]}}c_{\phi_0}$
from Section \ref{digression}. Here once we fix $\phi_\epsilon$, the
implied constant depends only on $F_{\Lambda_{\Gamma_\xi(d)}}$. But
since $\Gamma_\xi(d)$ is of finite index of $\Gamma$, it is
constant.  Hence the claim follows.
\end{proof}
\begin{corollary}There exists $\epsilon_0$ uniform over all square
free integer $d$ such that
$$|A_d(T)|= \frac{\cal
O^0_f(d)}{[\Gamma:\Gamma_\xi(d)]}(\chi+O(T^{\delta_\Gamma-\epsilon_0}))
,$$ where $\cal
O^0_f(d)=\sum_{\gamma_1\in\Gamma_\xi(d)\backslash\Gamma, \
f(\xi\gamma_1)=0(d)}1$  and
$\chi=\delta_\Gamma^{-1}c_{\phi_0}d_\epsilon||\xi||^{-\delta_\Gamma}T^{\delta_\Gamma}$.
\end{corollary}
\begin{proof}
$$|A_d(T)|=\sum_{n=0(d)}a_n(T)=\sum_{\scriptstyle \gamma\in{\text {Stab}}_\Gamma(\xi)\backslash
\Gamma \atop  \scriptstyle
f(\xi\gamma)=0(d)}w_T(\gamma)=\sum_{\scriptstyle
\gamma_1\in\Gamma_\xi(d)\backslash\Gamma \atop \scriptstyle
f(\xi\gamma_1)=0(d)}\sum_{\gamma\in\text{Stab}_\Gamma(\xi)\backslash\Gamma_\xi(d)}w_T(\gamma\gamma_1)$$
$$=\sum_{\gamma_1\in\Gamma_\xi(d)\backslash\Gamma,
\
f(\xi\gamma_1)=0(d)}\sum_{\gamma\in\text{Stab}_\Gamma(\xi)\backslash\Gamma_\xi(d)}\int_{G/K}\chi_{B_T}(\xi\gamma\gamma_1
\gamma_1^{-1}g)\phi_\epsilon(\gamma_1^{-1}g)d\mu(g)$$
$$=\sum_{\gamma_1\in\Gamma_\xi(d)\backslash\Gamma,
\
f(\xi\gamma_1)=0(d)}\int_{G/K}F^{\Gamma_\xi(d)}_T(g)\phi_\epsilon(\gamma_1^{-1}g)d\mu(g)$$
$$=\sum_{\gamma_1\in\Gamma_\xi(d)\backslash\Gamma,
\ f(\xi\gamma_1)=0(d)}\int_{\Gamma_\xi(d)\backslash
G/K}F^{\Gamma_\xi(d)}_T(g)\Phi_{\epsilon,\gamma_1}^{\Gamma_\xi(d)}(g)d\mu(g)$$
$$=\sum_{\gamma_1\in\Gamma_\xi(d)\backslash\Gamma,
\ f(\xi\gamma_1)=0(d)}\langle F_T^{\Gamma_\xi(d)},
\Phi_{\epsilon,\gamma_1}^{\Gamma_\xi(d)}\rangle_{L^2(\Gamma_\xi(d)\backslash
H^4_\bR)}.$$ By Proposition \ref{uniform},
$$|A_d(T)|=\sum_{\gamma_1\in\Gamma_\xi(d)\backslash\Gamma,
\ f(\xi\gamma_1)=0(d)}\frac{c_{\phi_0}d_\epsilon}{\delta_\Gamma
[\Gamma:\Gamma_\xi(d)]}||\xi||^{-\delta_\Gamma}T^{\delta_\Gamma}(1+O(T^{-\epsilon_0}))$$
$$= \frac{\cal
O^0_f(d)}{[\Gamma:\Gamma_\xi(d)]}(\chi+O(T^{\delta_\Gamma-\epsilon_0}))
,$$ where $\cal
O^0_f(d)=\sum_{\gamma_1\in\Gamma_\xi(d)\backslash\Gamma, \
f(\xi\gamma_1)=0(d)}1$ and
$\chi=\delta_\Gamma^{-1}c_{\phi_0}d_\epsilon||\xi||^{-\delta_\Gamma}T^{\delta_\Gamma}$,
hence the claim follows.
\end{proof}

\subsection{Digression to an algebraic group} Let $G$ be a semisimple
algebraic group defined over $\Q$. Denote $G(\F_p)$ a reduction of
$G$ to an algebraic group defined over $\F_p=\Z/p\Z$. In our case
$G=Spin(4,1)$, a double cover of $SO^0(4,1)$. Note that a stabilizer
of a point $\xi\in V=\{v=(x_1,\cdots,x_5)|Q(v)=0\}\setminus 0$ is a
connected $\Q$-group $H=NM$, so that $V=G/H$. By \cite{PR}
(Proposition 3.22), the reduction of $V$ to
$V(\F_p)=\{v=(v_1,\cdots,v_5)\in \F_p^5:Q(v)=0(p)\}$ is
$G(\F_p)/H(\F_p)$. By \cite{MVW}, since $\Gamma$ is Zariski dense,
the reduction map
$$\Gamma\ra G(\F_p)$$ is surjective. Note that $\Gamma_\xi(p)$ is a stabilizer of $\bar\xi$ in $V(\F_p)$ where $\bar\xi$
is a $p$-reduced image of $\xi$ in $V(\F_p)$. Since
the reduction of
$\Gamma$ is $G(\F_p)$,
$$[\Gamma,\Gamma_\xi(p)]=\#V(\F_p)\sim p^4 $$ by \cite{BS}.

Let $f_1(x_1,\cdots,x_5)=x_1,\ f_2(x_1,\cdots,x_5)=x_1x_2$. To
estimate $\cal
O^0_{f_1(p)}=\sum_{\gamma_1\in\Gamma_\xi(p)\backslash\Gamma, \
f_1(\xi\gamma_1)=0(p)}1$, if
$Q'(x_2,\cdots,x_5)=Q(0,x_2,\cdots,x_5)$ is a quadratic form whose
zero set is $W=\{(0,x_2,\cdots,x_5)\in V\}$, then
$$\cal O^0_{f_1(p)}=\# W(\F_p)\sim p^{3}$$ by \cite{BS}. Similarly
$$\cal O^0_{f_2(p)}=\# W(\F_p)\sim 2p^{3}.$$
Then
$$g_1(p)=\frac{\cal O^0_{f_1}(p)}{[\Gamma:\Gamma_\xi(p)]}\sim p^{-1},
 g_2(p)=\frac{\cal O^0_{f_2}(p)}{[\Gamma:\Gamma_\xi(p)]}\sim 2p^{-1}.$$
In general, if $f_k(x_1,\cdots,x_5)=x_1x_2\cdots x_k$, then
$g_k(p)\sim kp^{-1}$ for $k=1,\cdots, 5$.

\subsection{Asymptotic growth of number of spheres with prime
curvatures} Let $f_1(x_1,\cdots,x_5)=x_1,\
f_2(x_1,\cdots,x_5)=x_1x_2, f_k(x_1,\cdots,x_5)=x_1x_2\cdots x_k$ as
before and for square free integer $d$, let
$$g_i(d)=\frac{\cal O^0_{f_i}(d)}{[\Gamma:\Gamma_\xi(d)]}.$$
\begin{proposition}\label{prime}There exists a finite set $S$ of primes such that
\begin{enumerate}
\item for any square free integer $d=d_1d_2$ with no prime factors
in $S$ and for each $i=1,2$, $g_i(d_1d_2)=g_i(d_1)g_i(d_2)$
\item for any prime $p$ outside $S$, $g_i(p)\in (0,1)$ and
$g_1(p)=p^{-1}+O(p^{-q}),\ g_2(p)=2p^{-1}+O(p^{-q})$ for some $q\geq
0$.
\end{enumerate}

\end{proposition}
\begin{proof}The second claim is already shown. By \cite{MVW}, for a
large prime $p$, the reduction of $\Gamma$ is $G(\F_p)$. So let $S$
be the set of primes which are less than such a prime $p$.  Then for
$d=p_1\cdots p_k$ square free with $p_i\notin S$, the diagonal
reduction $$\Gamma\ra G(\Z/d\Z)\ra G(\F_{p_1})\times \cdots \times
G(\F_{p_k})$$ is surjective and it follows from Goursat's lemma that
$\Gamma$ surjects onto $G(\Z/d_1\Z)\times G(\Z/d_2\Z)$ for any
square free $d=d_1d_2$ without any prime factor in $S$. Hence for
$d=d_1d_2$ square free without prime factors in $S$, the orbit of
$\xi$ mod $d$ under $\Gamma$ is equal to the orbit of
$G(\Z/d_1\Z)\times G(\Z/d_2\Z)$ in $(\Z/d_1\Z)^5\times
(\Z/d_1\Z)^5$. The same thing is true for the orbit satisfying the
equation $f(\xi\gamma)=0(d_i)$. Therefore $g(d)=g(d_1)g(d_2)$. See
\cite{B,O}.
\end{proof}

Note that $a_n(T)=\sum_{\gamma\in\text{Stab}_\Gamma(\xi)\backslash
\Gamma,\ f_1(\xi\gamma)=n}w_T(\gamma)$ is a smooth counting function
for the number of vectors $(x_1,\cdots,x_5)$ in the orbit $\xi\cal
A^t$ of maximum norm bounded above by $T$ and $x_1=n$. We want to
obtain an asymptotic growth of the number $\pi^{\cal P}(T)$ of
spheres in an Apollonian packing $\cal P$ whose curvatures are prime
and less than $T$. The initial spheres with curvature
$\xi=(\xi_1,\cdots,\xi_5)$ is chosen so that $\xi_1<0$ which
corresponds to the largest sphere bounding all the other spheres,
and $\xi_2,\xi_3,\xi_4,\xi_5$ the smallest curvatures (largest
spheres) in the packing. Upon iteration, one obtains a smaller
sphere, hence larger curvature.

Since any sphere in the packing is obtained by the initial five
spheres corresponding to $\xi$ under the group $\cal A$, with the
maximum norm on $\bR^5$,
$$\pi^{\cal P}(T)\leq 4+\#\{\gamma\in \cal
A:||\xi\gamma^t||_{\text{max}}\ \text{is\ prime}<T\}$$
$$\ll\sum_{i=1}^5\#\{v\in \xi\cal A^t:||v||_{\text{max}}<T,\ v_i\
\text{is \ prime}\}.$$ Similarly the number $\pi^{\cal P}_2(T)$ of
twin prime curvatures less than $T$ is
$$\pi^{\cal P}_2(T)\ll\#\{\gamma\in\cal
A:||\xi\gamma^t||_{\text{max}}\ \text{is\ prime}<T,\ \text{one\
more\ entry\ of}\ \xi\gamma^t\ \text{is\
prime}\}$$$$\ll\sum_{i=1}^5\sum_{j\neq
i}\#\{v=(v_1,\cdots,v_5)\in\xi\cal A^t:||v||_{\text{max}}<T,\
v_i,v_j\ \text{primes}\}.$$

 This counting problem uses so-called a Selberg's sieve; using $S$
in Proposition \ref{prime}, and $g$ as one of $g_i$, a
multiplicative function $h$ on square free integers outside $S$ that
$h(p)=\frac{g(p)}{1-g(p)}$ for a prime $p$ outside $S$,
\begin{theorem}\label{number}(\cite{IK},\ Theorem 6.4)
Let $P$ be a finite product of distinct primes outside $S$ such that
for any square free $d| P$
$$|A_d(T)|=g(d)\chi+r_d(A(T)).$$  Then for any $D>1$,
$$S(A(T),P)=\sum_{(n,P)=1}a_n(T)\leq
\chi(\sum_{d<\sqrt D, d|P}h(d))^{-1}+\sum_{d<D,d
|P}\tau_3(d)|r_d(A(T))|,$$ where $\tau_3(d)$ denotes the number of
representations of $d$ as the product of three natural numbers.
\end{theorem}
From all these we obtain
\begin{theorem}\label{packing}
Given a bounded primitive integral Apollonian sphere packing $\cal
P$ in $\bR^3$,
$$\pi^{\cal P}(T)\ll\frac{T^{\delta_\Gamma}}{\log T}$$ and
$$\pi^{\cal P}_2(T)\ll\frac{T^{\delta_\Gamma}}{(\log T)^2}.$$
Generally if $\pi^{\cal P}_k(T)$ denotes the number of $k$-mutually
tangent spheres whose curvatures are prime numbers less than $T$,
then
$$\pi^{\cal P}_k(T)\ll\frac{T^{\delta_\Gamma}}{(\log T)^k}$$ for
$k\leq 5$.
\end{theorem}
\begin{proof}Since $$|A_d(T)|= \frac{\cal
O^0_f(d)}{[\Gamma:\Gamma_\xi(d)]}(\chi+O(T^{\delta_\Gamma-\epsilon_0}))=g_1(d)\chi+r_d(A(T))
,$$ and since $\frac{\cal O^0_f(d)}{[\Gamma:\Gamma_\xi(d)]}<1$,
$r_d(A(T))\ll T^{\delta_\Gamma-\epsilon_0}$. For any $\epsilon_1$,
$\sum_{d<D}\tau_3(d)<1^3+2^3+\cdots+ D^3<< D^{3+\epsilon_1}$, hence
$$\sum_{d<D,d
|P}\tau_3(d)|r_d(A(T))|\ll
D^{3+\epsilon_1}T^{\delta_\Gamma-\epsilon_0}\ll
T^{\delta_\Gamma}/\log T$$ if $D=T^{\epsilon_0/4}$. Hence take $P$
to be the product of all primes less than $T^{\epsilon_0/4}$ outside
$S$. Denote $\mu(n)=1$ if $n$ is square-free and $0$ otherwise. Also
let $s(m)$ denote the largest square-free number dividing $m$. Then
for $h=\frac{g_1}{1-g_1}$, one can deduce that (\cite{IK}, Section
6.6)
$$\sum_{d<\sqrt D, d|P}h(d)\sim\sum_{d<\sqrt D, d|P}\Pi_{p|d}\frac{p^{-1}}{1-p^{-1}}$$
$$=\sum_{d<\sqrt D, d|P}d^{-1}\Pi_{p|d}\frac{1}{1-p^{-1}}=\sum_{d<\sqrt D,
d|P}d^{-1}\Pi_{p|d}(1+\frac{1}{p}+\frac{1}{p^2}+\cdots)$$
$$= \sum_{d=p_1\cdots p_k<\sqrt D,d|P} (\frac{1}{d}+  \frac{1}{dp_1}+\frac{1}{dp_1^2}+\cdots)(1+\frac{1}{p_2}+\frac{1}{p_2^2}+
\cdots)\cdots(1+\frac{1}{p_k}+\frac{1}{p_k^2}+\cdots).
$$ But since the numbers appearing in the denominator is of the form
$m=p_1^{\alpha_1}\cdots p_k^{\alpha_k}$ and since $S$ is a fixed
finite set of primes less than some number, and $P$ is the product
of distinct primes less than $D$ outside $S$, when $D$ tends to
infinity, the above sum is
$$\gg \sum_{d<\sqrt
D}({\mu(d)}\sum_{m,s(m)=d}\frac{1}{m})=\sum_{m,s(m)<\sqrt
D}\frac{1}{m}$$
     $$\gg\sum_{m<\sqrt D}m^{-1}\gg\log D\gg\log T.$$ Hence Theorem
\ref{number} gives
$$S(A(T),P)\ll\frac{T^{\delta_\Gamma}}{\log T}.$$

Since there are only 5 coordinates, we may assume that for any
$v=(v_1,\cdots,v_5)\in\xi\cal A^t$ $v_1$ is the largest coordinate
among $v_i$. Therefore by Equation (\ref{squeeze}) for the product
$P$ of distinct primes less than $T^{\epsilon_0/4}$ outside $S$,
$$\frac{T^{\delta_\Gamma}}{\log
T}\gg
S(A((1+\epsilon)T),P)+(T^{\epsilon_0/4})^{\delta_\Gamma}=\sum_{(n,P)=1}a_n((1+\epsilon)T)+(T^{\epsilon_0/4})^{\delta_\Gamma}
$$$$\gg\#\{v=(x_1,\cdots,x_5)\in \xi\cal A^t:||v||_{\text{
max}}<T, (x_1,P)=1\}+(T^{\epsilon_0/4})^{\delta_\Gamma}$$
$$\gg (T^{\epsilon_0/4})^{\delta_\Gamma}+\#\{v=(x_1,\cdots,x_5)\in \xi\cal A^t:||v||_{\text{
max}}<T, x_1\ \text{prime}>T^{\epsilon_0/4}\}
$$
$$\gg 5\#\{v=(x_1,\cdots,x_5)\in \xi\cal A^t:||v||_{\text{
max}}<T, x_1\ \text{prime}\}$$
$$\gg\#\{v=(x_1,\cdots,x_5)\in \xi\cal A^t:||v||_{\text{
max}}<T, \text{some\ $v_i$\
 prime}\}\gg\pi^{\cal P}(T).$$

Let $\omega(d)$ denote the number of  distinct prime factors of $d$
and ${\cal D}(d)$ the number of positive divisors of $d$ including
$1$. Then ${\cal D}(d)=2^{\omega(d)}$ for square free integer $d$.
For $g_2$, considering $S$ is a finite set and $d|P$ implies that
$d$ is a square free, as before
$$\sum_{d<\sqrt D, d|P}h(d)\sim\sum_{d<\sqrt D, d|P}\Pi_{p|d}\frac{2p^{-1}}{1-2p^{-1}}$$$$=\sum_{d<\sqrt D, d|P}
d^{-1}2^{\omega(d)}\Pi_{p|d}\frac{1}{1-2p^{-1}}\gg\sum_{d<\sqrt
D}\mu(d) d^{-1}{\cal
D}(d)\Pi_{p|d}(1+\frac{1}{p}+\frac{1}{p^2}+\cdots)$$
$$=\sum_{d<\sqrt D}\mu(d) \cal D(d)\sum_{m,s(m)=d}\frac{1}{m}. $$
If $m=p_1^{\alpha_1}\cdots p_k^{\alpha_k}$ for distinct primes $p_i$
then $s(m)=d=p_1\cdots p_k$. Then $\cal D(d)=2^k,\ \cal
D(m)=(\alpha_1+1)\cdots (\alpha_k+1)$. Hence
$$\frac{\cal D(s(m))}{d}> \frac{\cal D(m)}{m}.$$
For each $m\leq \sqrt D$, if $s(m)=d\leq m$, then $\frac{\cal
D(d)}{d}$ appears in $$\sum_{d<\sqrt D}\mu(d) \cal
D(d)\sum_{m,s(m)=d}\frac{1}{m}.$$ Hence
$$\sum_{d<\sqrt D}\mu(d) \cal D(d)\sum_{m,s(m)=d}\frac{1}{m}>
\sum_{m<\sqrt D}\frac{\cal D(m)}{m}.$$

 Note by \cite{MV} (page 94),
$$\sum_{n=1}^N\frac{\cal D(n)}{n}=\frac{1}{2}(\log N)^2+O(\log N).$$
From this
$$\sum_{d<\sqrt D, d|P}h(d)\gg (\log D)^2\gg (\log T)^2.$$
$$\frac{T^{\delta_\Gamma}}{(\log T)^2}\gg S(A((1+\epsilon)T),P)=\sum_{(n,P)=1}a_n((1+\epsilon)T)
$$$$\gg\#\{v=(x_1,\cdots,x_5)\in \xi\cal A^t:||v||_{\text{
max}}<T, (x_1x_2,P)=1\}+(T^{\epsilon_0/4})^{\delta_\Gamma}$$
$$\gg\#\{v=(x_1,\cdots,x_5)\in \xi\cal A^t:||v||_{\text{
max}}<T, x_1,x_2\ \text{prime}>
T^{\epsilon_0/4}\}+(T^{\epsilon_0/4})^{\delta_\Gamma}$$$$\gg\pi_2^{\cal
P}(T).$$ For general $k\leq 5$, it suffices to estimate
$$\sum_{d<\sqrt D,
d|P}\Pi_{p|d}\frac{kp^{-1}}{1-kp^{-1}}$$$$=\sum_{d<\sqrt D, d|P}
d^{-1}k^{\omega(d)}\Pi_{p|d}\frac{1}{1-kp^{-1}}\gg\sum_{d<\sqrt
D}\mu(d)
d^{-1}k^{\omega(d)}\Pi_{p|d}(1+\frac{1}{p}+\frac{1}{p^2}+\cdots)$$
$$=\sum_{d<\sqrt D}\mu(d) k^{\omega(d)}\sum_{m,s(m)=d}\frac{1}{m}. $$
If $m=p_1^{\alpha_1}\cdots p_k^{\alpha_k}$ for distinct primes $p_i$
then $s(m)=d=p_1\cdots p_k$.  Hence
$$\frac{k^{\omega(d)}}{d}> \frac{k^{\omega(m)}}{m}.$$
For each $m\leq \sqrt D$, if $s(m)=d\leq m$, then
$\frac{k^{\omega(d)}}{d}$ appears in $$\sum_{d<\sqrt D}\mu(d)
k^{\omega(d)}\sum_{m,s(m)=d}\frac{1}{m}.$$ Hence
$$\sum_{d<\sqrt D}\mu(d) k^{\omega(d)}\sum_{m,s(m)=d}\frac{1}{m}>
\sum_{m<\sqrt D}\frac{ k^{\omega(m)}}{m}\gg (\log D)^k,$$ where the
last inequality follows from the following theorem \cite{Ten}, and
communicated to us by Pieter Moree.
\begin{theorem}$$\sum_{n\leq x} k^{\omega(n)}=c_k(0)x(\log
x)^{k-1}+O(x(\log x)^{k-2}),$$ where
$$c_k(0)=\frac{1}{(k-1)!}\Pi_p(1+\frac{k}{p-1})(1-\frac{1}{p})^k.$$
\end{theorem}
\end{proof}
Indeed, using a finer version (Theorem 11.13 of \cite{FI}) of Sieve
method of a multiplicative function
$g(p)=\frac{k}{p}+O(p^{-1-\delta})$
\begin{theorem}There exists a real number $\beta(k)>0$ such that
$$(f(s)+O(\log D)^{-1/6})|A(T)|\Pi_{p|P(z)}(1-g(p))+ R(D)
\leq S(A(T),P(z))$$$$\leq(F(s)+O((\log
D)^{-1/6})|A(T)|\Pi_{p|P(z)}(1-g(p))+ R(D)$$ where $z=D^{1/s}$ with
$s>\beta(k)$, and $P(z)$ is the product of distinct primes less than
$z$ outside $S$, where $F(s)>0$ and $f(s)>0$ are certain functions
of $s\geq 0$, depending on $k$, defined as solutions of explicit
differential-difference equations, such that
$$\lim_{s\ra\infty}f(s)=\lim_{s\ra\infty}F(s)=1,$$
and where     $$R(D)=\sum_{d<D} |r_d(A(T))|.$$ In both upper and
lower bounds, the implied constant depends only on $k$ and on the
constants in the asymptotic of $g(p)$.
\end{theorem}
we can try to give a lower bound for our counting problem. Note that
the number $\pi^\cal P(T)$ of spheres whose curvature is prime less
than $T$ is equal to $$\#\{v=(x_1,\cdots,x_5)\in \xi\cal
A^t:||v||_{\text{ max}}={\text {prime}}<T \}.$$ But $a_n(T)$
approximately counts the number of orbit vectors with maximum norm
less than $T$ and whose first coordinate is $n$. Hence as long as
the first coordinate is prime $p$, both $(p,v_2,v_3,v_4,v_5)$ and
$(p,v_2',v_3',v_4',v_5')$ will be counted even though none of
$v_2',v_3',v_4',v_5'$ is prime. This means that $a_n(T)$ with $n$
prime, will count the same sphere as many times as it appears in the
orbit $\xi\cal A^t$. Another difficulty is that
$$\#\{v=(x_1,\cdots,x_5)\in \xi\cal
A^t:||v||_{\text{ max}}<T, T^{\epsilon_0/4}< x_1\ \text{prime}< T
\}$$ is not equivalent to
$$\#\{v=(x_1,\cdots,x_5)\in \xi\cal
A^t:||v||_{\text{ max}}<T, (x_1, \Pi_{\text{prime}\ p<
T^{\frac{\epsilon_0}{4}}}p)=1\}.$$ To make them comparable, we pose
the condition that $x_1$ cannot be written as a product of more than
$r$ primes $>T^{\epsilon_0/4}$ where $r$ is chosen as the first
positive integer to satisfy $T^{\frac{r\epsilon_0}{4}}>T$. People
call it $r$-almost prime.

These things complicate the problem for the lower bound. But if we
allow this over-counting with extra assumptions, one can give a
lower bound.
\begin{theorem}\label{almost}Let $\pi_k^{\cal P}(T)^r$ denote the number of 5
spheres kissing each other (i.e. the number of orbits) among whose
at least $k$ curvatures are $r$-almost primes and all of whose
curvatures are less than $T$ where $r$ is  the first positive
integer larger than $\frac{4s_0}{\epsilon_0}$ for a fixed large
$s_0$. For any $k\leq 5$,
$$\frac{T^{\delta_\Gamma}}{(\log T)^k}\ll \pi_k^{\cal P}(T)^r.$$
\end{theorem}
\begin{proof}We know already that $|A(T)|\sim T^{\delta_\Gamma}$.
Note also that since $r_d(A(T))\ll
T^{\delta_\Gamma-\epsilon_0}$,$$R(D)=\sum_{d<D} |r_d(A(T))|\leq
DT^{\delta_\Gamma-\epsilon_0}.$$ Hence if $D=T^{\epsilon_0/4}$,
$R(D)\ll \frac{T^{\delta_\Gamma}}{(\log T)^k}$. Now we estimate
$\Pi_{p|P(z)}(1-g(p))$ using Mertens formula, which is communicated
to us by Pieter Moree. Put $A_k(p)=1-k/p$ if $p>k$ and $A_k(p)=1$
otherwise. Put $B_k=\prod_{p\ge 2}A_k(p)(1-1/p)^{-k}$. Let $\gamma$
denote Euler's constant.\hfil\break We claim that {\it as $x$ tends
to infinity we have $\prod_{p\le x}A_k(p)\sim
B_ke^{-k\gamma}\log^{-k}x$.} Write
$$\prod_{p\le x}A_k(p)=\prod_{p\le x}(1-1/p)^k\prod_{p\le x}A_k(p)(1-1/p)^{-k}.$$
Note that $\log A_k(p)-k\log(1-1/p)=O(k^2/p^2)$ as $p$ tends to
infinity. It thus follows that the latter product converges to $B_k$
as $x$ tends to infinity. Noting that $$\sum_{p>x}|\log
A_k(p)-k\log(1-1/p)|=O(\sum_{p>x}{k^2\over p^2})=O({k^2\over x}),$$
we find that
$$\prod_{p\le x}A_k(p)=\prod_{p\le x}(1-1/p)^kB_k(1+O({k^2\over x})).$$
Using the Mertens theorem  that
$$\prod_{p\le x}(1-1/p)\sim {e^{-\gamma}\over \log x},$$
the claim then follows.

Now take $D=T^{\epsilon_0/4}$ and let $s$ large enough and
$\epsilon\ra 0$ to get for $z=T^{\frac{\epsilon_0}{4s}}$ and the
product $P=P(z)$ of distinct primes less than
$T^{\frac{\epsilon_0}{4s}}$ outside $S$
$$\frac{T^{\delta_\Gamma}}{(\frac{\epsilon_0}{4s})^k(\log T)^k}\ll (f(s)+O(\log D)^{-1/6})|A(T)|\Pi_{p|P(z)}(1-g(p))+ R(D)$$$$\ll S(A((1-\epsilon)T),P)$$$$=\sum_{(n,P)=1}a_n((1-\epsilon)T)
\ll\#\{v=(x_1,\cdots,x_5)\in \xi\cal A^t:||v||_{\text{ max}}<T,
(x_1x_2 \cdots x_k,P)=1\}$$ Since there are only 5 coordinates,
there are only $5Ck$ choices of $k$ coordinates out of 5,
$$\ll\#\{v=(x_1,\cdots,x_5)\in \xi\cal A^t:||v||_{\text{ max}}<T^{\frac{\epsilon_0}{4s}}\}+$$$$5Ck\#\{v=(x_1,\cdots,x_5)\in
\xi\cal A^t:||v||_{\text{ max}}<T, (x_1x_2 \cdots x_k,P)=1,\exists
1\leq i \leq k, x_i>T^{\frac{\epsilon_0}{4s}}\}$$
$$\ll \#\{v=(x_1,\cdots,x_5)\in \xi\cal A^t:||v||_{\text{ max}}<T,
(x_1x_2 \cdots x_k,P)=1,\
x_i>T^{\frac{\epsilon_0}{4s}},i=1,\cdots,k\}$$$$+(T^{\frac{\epsilon_0}{4s}})^{\delta_\Gamma}$$
$$\ll \#\{v=(x_1,\cdots,x_5)\in \xi\cal A^t:||v||_{\text{
max}}<T, T>x_1,x_2,\cdots,x_k\ r\text{-almost
prime}>T^{\frac{\epsilon_0}{4s}} \}$$$$
+(T^{\frac{\epsilon_0}{4s}})^{\delta_\Gamma}\ll \pi_k^{\cal
P}(T)^r.$$
\end{proof}
Note that to get an upper bound, we need to reverse the inequality
$$ \#\{v=(x_1,\cdots,x_5)\in \xi\cal A^t:||v||_{\text{ max}}<T,
(x_1x_2 \cdots x_k,P)=1,\
x_i>T^{\frac{\epsilon_0}{4s}}\}+(T^{\frac{\epsilon_0}{4s}})^{\delta_\Gamma}$$
$$\ll\#\{v=(x_1,\cdots,x_5)\in \xi\cal A^t:||v||_{\text{
max}}<T, T>x_1,x_2,\cdots,x_k\ r\text{-almost
prime}>T^{\frac{\epsilon_0}{4s}} \},$$ which is not obvious to us.
\begin{corollary}Given a bounded primitive integral Apollonian sphere packing $\cal
P$ in $\bR^3$, $\pi_5^{\cal P}(T)^r$ denote the number of 5 spheres
kissing each other whose  curvatures are $r$-almost primes less than
$T$ where $r$ is a fixed positive integer depending only on
Apollonian packing. Then
$$\frac{T^{\delta_\Gamma}}{(\log T)^5}\ll \pi_5^{\cal P}(T)^r.$$
\end{corollary}

{\bf Acknowledgements} The author thanks Pierre Pansu for numerous
conversations and his help. He also thanks E. Breuillard for series
of lectures  about approximate groups, A. Gamburd for the
discussions about uniform spectral gap and H. Oh for useful comments
about the first draft of this paper. He  owes P. Moree much
gratitude for his help in number theory in the last section.

\space  Inkang Kim\\
School of Mathematics\\
     KIAS, Hoegiro 85, Dongdaemun-gu\\
     Seoul, 130-722, Korea\\
     \texttt{inkang\char`\@ kias.re.kr}

\end{document}